\documentclass{siamart220329}
\usepackage{changes}

\usepackage[english]{babel}
\usepackage{booktabs}

\usepackage{cite}
\makeatletter
\newcommand{\citecomment}[2][]{\citen{#2}#1\citevar}
\newcommand{\citeone}[1]{\citecomment{#1}}
\newcommand{\citetwo}[2][]{\citecomment[,~#1]{#2}}
\newcommand{\citevar}{\@ifnextchar\bgroup{;~\citeone}{\@ifnextchar[{;~\citetwo}{]}}}
\newcommand{\citefirst}{\@ifnextchar\bgroup{\citeone}{\@ifnextchar[{\citetwo}{]}}}

\makeatother

\usepackage{amssymb}
\usepackage{mathtools}

\DeclarePairedDelimiter{\abs}{\lvert}{\rvert}
\providecommand{\norm}[1]{\lVert#1\rVert_2}
\newcommand*{\dif}{\mathop{}\!\mathrm{d}}
\newcommand*{\Tra}[1]{{#1}^{\mathsf{T}}}
\newcommand*{\complns}{\mathbb{C}} 

\DeclarePairedDelimiter{\opfences}{(}{)}
\newcommand*{\real}{\operatorname{Re}\opfences}

\DeclarePairedDelimiter{\opbrakets}{\{}{\}}
\newcommand*{\Lp}{\operatorname{\mathcal{L}}\opbrakets}

\newcommand{\Cn}{\mathbb{C}^{n}}
\newcommand{\Cnn}{\mathbb{C}^{n \times n}}
\newcommand{\Cm}{\mathbb{C}^{m}}
\newcommand{\Cmm}{\mathbb{C}^{m \times m}}
\newcommand{\Cnm}{\mathbb{C}^{n \times m}}
\newcommand{\Kry}{{\cal K}}
\DeclareMathOperator{\Span}{span}
\DeclareMathOperator{\spec}{spec}
\newsiamremark{remark}{Remark}

\usepackage{algpseudocode}
\usepackage{pgfplots}
\pgfplotsset{compat=1.3}
\usepgfplotslibrary{groupplots}
\usepackage[caption=false]{subfig}
\definecolor{mycolor1}{rgb}{0.00000,0.44700,0.74100} 
\definecolor{mycolor2}{rgb}{0.85000,0.32500,0.09800} 
\definecolor{mycolor3}{rgb}{0.92900,0.69400,0.12500} 

\newcommand*{\titleforLaplace}{Laplace op.\ $A_\text{L}$ ($m = 50$)}
\newcommand*{\titleforConvdiff}{conv.-diff.\ op.\ $A_\text{CD}$ ($m=20$)}
\newcommand*{\titleforLaplacewom}{Laplace op.\ $A_\text{L}$}
\newcommand*{\titleforConvdiffwom}{conv.-diff.\ op.\ $A_\text{CD}$}

\newcommand{\corrected}[1]{#1}

\pgfplotsset{laplace1/.style={
    color=mycolor1, 
    line width=1.5pt, 
    mark=o, 
    mark options={solid, mycolor1}} }

\pgfplotsset{laplace2/.style={
    color=mycolor1, 
    line width=1.5pt,
    mark=o, 
    only marks, 
    mark options={solid, mycolor1}} }

\pgfplotsset{laplace3/.style={
    color=mycolor1,  
    line width=1.5pt} }

\pgfplotsset{stieltjes1/.style={
    color=mycolor2, 
    line width=1.5pt, 
    dashed, 
    mark=square, 
    mark options={solid, mycolor2}} }

\pgfplotsset{stieltjes2/.style={
    color=mycolor2, 
    line width=1.5pt,
    mark=square, 
    only marks, 
    mark options={solid, mycolor2}} }

\pgfplotsset{stieltjes3/.style={
    color=mycolor2,  
    line width=1.5pt,
    dashed} }

\pgfplotsset{tolline/.style={
    color=black,  
    line width=1.5pt,
    dotted} }
    
\pgfplotsset{twopass1/.style={
    color=mycolor3, 
    line width=1.5pt, 
    dashdotted,
    mark=diamond, 
    mark options={solid, mycolor3}} } 

\pgfplotsset{twopass2/.style={
    color=mycolor3, 
    line width=1.5pt,
    mark=diamond, 
    only marks, 
    mark options={solid, mycolor3}} }
    
\pgfplotsset{twopass3/.style={
    color=mycolor3,  
    line width=1.5pt,
    dashdotted} }

\title{Krylov subspace restarting for matrix Laplace transforms}
\author{%
A.\ Frommer\thanks{School of Mathematics and Natural Sciences, Bergische Universit\"at Wuppertal, 42097 Wuppertal, Germany, \email{\{frommer,kkahl,marcel,tsolakis\}@uni-wuppertal.de}}
\and
K.\ Kahl${}^\ast$\!\!\!
\and
M.\ Schweitzer${}^\ast$\!\!\!
\and
M.\ Tsolakis${}^\ast$ 
}

\begin{document}

\maketitle\pagestyle{myheadings} \thispagestyle{plain}
\markboth{A.~FROMMER, K.~KAHL, M.~SCHWEITZER, AND M.~TSOLAKIS}{KRYLOV SUBSPACE RESTARTING FOR MATRIX LAPLACE TRANSFORMS}

\begin{abstract}
A common way to approximate $F(A)b$---the action of a matrix function on a vector---is to use the Arnoldi approximation. Since a new vector needs to be generated and stored in every iteration, one is often forced to rely on restart algorithms which are either not efficient, not stable or only applicable to restricted classes of functions. 
We present a new representation of the error of the Arnoldi iterates if the function $F$ is given as a Laplace transform. Based on this representation we build an efficient and stable restart algorithm. In doing so we extend earlier work for the class of Stieltjes functions which are special Laplace transforms.
We report several numerical experiments including comparisons with the restart method for Stieltjes functions.
\end{abstract}


\begin{AMS}
65F60, 
65F50, 
44A10, 
65D30, 
65D07  
\end{AMS}

{\vspace{.05in}\footnotesize
\parindent .2in{\upshape\bfseries MATLAB package }\ignorespaces 
available at \url{https://github.com/MaTso7/laplace_restarting}
\par\vspace{.1in}}

\section{Introduction}
Computing $F(A)b$, the action of a matrix function $F(A) \in \Cnn$ on a vector $b \in \Cn$, is an important task in many scientific computing applications, including exponential integrators for differential equations~\cite{HochbruckOstermann2010}, network analysis~\cite{BenziBoito2020}, theoretical particle physics~\cite{VanDenEshofFrommerLippertSchillingVanDerVorst2002}, machine learning~\cite{PleissJankowiakErikssonDamleGardner2020} and many others. In these applications, the matrix $A \in \Cnn$ is typically very large and sparse, so that explicitly forming $F(A)$---which is a dense matrix in general, irrespective of the sparsity of $A$---is not feasible with regard to both  complexity and memory requirements. Therefore, one has to resort to \emph{iterative methods} that directly approximate the vector $F(A)b$. The most widely used classes of algorithms for this task are polynomial~\cite{Saad1992,DruskinKnizhnerman1989} and rational~\cite{Guettel2010,DruskinKnizhnerman1998,vdEH06,MasseiRobol2021} Krylov subspace methods. 

While rational Krylov methods can greatly outperform polynomial methods when they are applicable (e.g., when shifted linear systems with $A$ can be efficiently solved by a sparse direct solver), there are situations in which polynomial methods are superior: The size and sparsity pattern of $A$ might make the direct solution of shifted linear systems infeasible, or $A$ might only be implicitly available through a routine that returns the result of the matrix-vector product. In these cases, combining an ``outer'' rational Krylov method with an ``inner'' polynomial iterative solver for linear systems is in general not advisable (unless a \emph{very} efficient preconditioner is available); see~\cite{GuettelSchweitzer2021}. Thus---despite the tremendous theoretical and algorithmic advances in the area of rational Krylov methods in recent years---polynomial Krylov methods are still of utmost importance, in particular for large-scale applications. 

However, for very large matrix sizes $n$, polynomial Krylov methods present challenges on their own: 
Their backbone is the Arnoldi process~\cite{Arnoldi1951} (which reduces to the short-recurrence Lanczos method~\cite{Lanczos1950} when $A$ is Hermitian), which computes a nested orthonormal basis $v_1,\dots,v_m$ of the Krylov subspace
$$\Kry_m(A,b) \coloneqq \Span\{b, Ab, \dots, A^{m-1}b\}.$$
Collecting the basis vectors in $V_m = [v_1\mid\dots\mid v_m] \in \Cnm$ and the coefficients from the orthogonalization process in an upper Hessenberg matrix $H_m \in \Cmm$, one has the \emph{Arnoldi relation}
\begin{equation}\label{eq:arnoldi_relation}
    AV_m = V_mH_m + h_{m+1,m}v_{m+1}e_m^T,
\end{equation}
where $e_m \in \Cm$ denotes the $m$th canonical unit vector. Given~\eqref{eq:arnoldi_relation}, one obtains the \emph{Arnoldi approximation} $f_m \in \Kry_m(A,b)$ for $F(A)b$ by projecting the original problem onto the Krylov space, i.e.,
\begin{equation}\label{eq:arnoldi_approximation}
    f_m \coloneqq V_mF(V_m^HAV_m)V_m^Hb = \|b\|_2V_mF(H_m)e_1,
\end{equation}
where $\|\cdot\|_2$ denotes the Euclidean norm and $e_1 \in \Cm$ is the first canonical unit vector. A large computational burden associated with forming the approximation~\eqref{eq:arnoldi_approximation} is that it requires storing all vectors of the orthonormal basis $v_1,\dots,v_m$. This can quickly surpass the available memory if $n$ is very large. This is also the case when $A$ is Hermitian, in contrast to the situation for linear systems, where the short recurrence for the basis vectors translates into a short recurrence for the iterates in, for example, the conjugate gradient method~\cite{HestenesStiefel1952}. Thus, without appropriate countermeasures, the approximation accuracy that is reachable by a Krylov method might be limited by the available memory. Additionally, when $A$ is non-Hermitian, the orthogonalization of basis vectors becomes more and more expensive with growing $m$, rendering Krylov methods much less efficient when a large number of iterations is required.

A remedy 
is \emph{restarting}: After a fixed (small) number $m_{\max}$ of iterations, the Arnoldi approximation $f_{m_{\max}}$ is formed. Then, the matrices $V_m$ and $H_m$ are discarded, and a new Arnoldi process is started for approximating the remaining \emph{error}
\begin{equation}\label{eq:error_general}
    \varepsilon_{m_{\max}} \coloneqq F(A)b - f_{m_{\max}}.
\end{equation}
Clearly, approximating~\eqref{eq:error_general} by the Arnoldi method 
requires that it can again be written as the action of a matrix 
function on a vector. While this is evidently the case when solving 
linear systems (i.e., for $F(s) = s^{-1}$), where the error fulfills 
the residual equation $A\varepsilon_{m_{\max}} = r_{m_{\max}}$ with 
$r_{m_{\max}} = b-Af_{m_{\max}}$, the situation is much more difficult 
for general functions $F$. Over the last fifteen years, numerous 
publications have investigated how to transfer the restarting approach 
to general matrix functions~\cite{FrommerGuettelSchweitzer2014a,Schweitzer2016,TalEzer2007,EiermannErnst2006,AfanasjewEtAl2008a,AfanasjewEtAl2008b,IlicTurnerSimpson2010}. All of these approaches have in common 
that they are either only applicable to certain classes of functions 
(i.e., they are not fully general) or they have shortcomings with 
regard to numerical stability or computational efficiency (e.g., the 
required work grows from one restart cycle to the next). In this work, 
we add to the current state of the art in restarted Krylov methods by 
explaining---both theoretically and algorithmically---how restarts are 
possible when $F$ results from the Laplace transform of some function 
$f$. This covers a large class of practically relevant functions and 
thus significantly extends the scope in which restarting is possible 
in a numerically efficient and stable way. While the approach based on 
the Cauchy integral formula presented 
in~\cite{FrommerGuettelSchweitzer2014a} is in theory applicable to any 
analytic function, and thus also to Laplace transforms, its practical 
implementation requires choosing a suitable contour $\Gamma$ in the 
complex plane and a corresponding quadrature rule. When choosing 
these, one needs knowledge on the position of 
the spectrum of $A$ in the complex plane 
and has to take specifics of the function $F$ to be approximated into 
account as otherwise the quadrature rule on $\Gamma$ might converge 
very slowly or become unstable. 
In contrast, the method we propose here works for general Laplace 
transforms, only employing simple, general purpose Gauss-Kronrod 
quadrature rules on the positive real axis. The only requirement is 
that the field of values lies within the region of absolute 
convergence of the Laplace transform, the typical case being that $A$ 
is positive definite (not necessarily Hermitian) in the sense that
    $\real{x^H A x} > 0$ for $x \neq 0$,
and the Laplace transform converges absolutely in the right half plane.

The remainder of this paper is organized as follows. Important basic material is covered in \cref{sec:basics}. In particular, we review the related Arnoldi restarting approach for Stieltjes matrix functions, and we collect some basic facts about the Laplace transform. In \cref{sec:laplace_theory}, we lay the necessary theoretical foundations for restarting for matrix Laplace transforms and explain how it relates to other restart approaches. We explain in \cref{sec:extensions} how the method can be extended to two related classes of functions.
\Cref{sec:laplace_implementation} is devoted to several 
implementation aspects that are crucial for making the method feasible in practice. In \cref{sec:experiments}, we illustrate the performance of our method and compare it to existing alternatives in a series of numerical experiments on both academic and real-world benchmark problems. Concluding remarks are given in \cref{sec:conclusions}.

\section{Basics}\label{sec:basics}
In this section, we review some basic material on which we build in later sections.

\subsection{The restarted Arnoldi method for Stieltjes functions}\label{subsec:restart_stieltjes}
We start by going into the details of the quadrature-based restarting approach for Stieltjes matrix functions from~\cite{FrommerGuettelSchweitzer2014a,Schweitzer2016}, as it is closely related to the present work; see also \Cref{cor:error_laplace_to_stieltjes} below.

Given a maximum Krylov subspace size $m$ (dictated, e.g., by the available memory), the basic idea of the restarted Arnoldi method for $F(A)b$ is to compute a sequence of (hopefully) more and more accurate approximations $f_m^{(1)}, f_m^{(2)}, f_m^{(3)}, \dots$, where $f_m^{(1)} = f_m = \|b\|_2V_mF(H_m)e_1$ is the usual $m$-step Arnoldi approximation~\eqref{eq:arnoldi_approximation} and further iterates (belonging to the $k$th \emph{Arnoldi cycle}) are defined via the recurrence 
\begin{equation}\label{eq:restarted_arnoldi_recurrence}
    f_m^{(k+1)} = f_m^{(k)} + d_m^{(k)}, \quad k \geq 1
\end{equation}
where $d_m^{(k)}$ is the $m$-step Arnoldi approximation for the error $\varepsilon_m^{(k)} = F(A)b - f_m^{(k)}$. Clearly, the preceding statement only makes sense if $\varepsilon_m^{(k)} = F^{(k+1)}(A)b^{(k+1)}$ is again the action of a matrix function $F^{(k+1)}(A)$ on a \corrected{normalized} vector $b^{(k+1)}$. Note that in general $F^{(k)} \neq F$, and, as it turns out, in all known approaches one finds $b^{(k)} = v_{m+1}^{(k-1)}$, the $(m+1)$st Arnoldi vector of the previous cycle. A generic algorithmic description is given in \cref{alg:generic_arnoldi}.
\begin{algorithm}
\caption{Restarted Arnoldi method for $F(A)b$ from \cite{EiermannErnst2006} as presented in \cite{FrommerGuettelSchweitzer2014a}.}\label{alg:generic_arnoldi}
\begin{algorithmic}[1]
\State Compute the Arnoldi decomposition $AV_m^{(1)} = V_m^{(1)}H_m^{(1)} + h_{m+1,m}^{(1)}v_{m+1}^{(1)}\Tra{e}_m$ \Statex\hspace{\algorithmicindent}with respect to $A$ and $b$.
\State Set $f_m^{(1)} = \norm{b}V_m^{(1)}F(H_m^{(1)})e_1$.
\For{$k = 2, 3, \dots$ until convergence}
    \State Determine the error function $F^{(k)}$ s.t.\ $\varepsilon^{(k-1)} = F^{(k)}(A)v^{(k-1)}_{m+1}$.
    \State Compute the Arnoldi decomposition $AV_m^{(k)} = V_m^{(k)}H_m^{(k)} + h_{m+1,m}^{(k)}v_{m+1}^{(k)}\Tra{e}_m$ \Statex\hspace{\algorithmicindent}\hspace{\algorithmicindent}with respect to $A$ and $v_{m+1}^{(k-1)}$.
    \State Set $f_m^{(k)} = f_m^{(k-1)} + \norm{b}V_m^{(k)}F^{(k)}(H_m^{(k)})e_1$.
\EndFor
\end{algorithmic}
\end{algorithm}

In early work on restarting for matrix functions, $F^{(k)}, k > 1$ is characterized in terms of divided differences of $F$. 
\corrected{We have
\[
F^{(2)}(s) = \| b\|_2  \prod_{i = 1}^m h_{i+1,i}  [D_{w_m}F](s),
\]
where $[D_{w_m}F]$ denotes the $m$-th divided difference of $F$ with respect to interpolation nodes which are the eigenvalues of $H_m^{(1)}$; see~\cite{EiermannErnst2006,TalEzer2007,IlicTurnerSimpson2010}. 
This expression can be iterated to obtain (iterated) divided difference representations for the error functions $F^{(k)}(t), k >2$.
The divided difference representations might be helpful to keep in mind when following the theory developed in this paper. They are, however, only marginally useful when it comes to computation, since they suffer from numerical instability.   
}

\corrected{This is why} error function representations for rational functions in partial fraction form~\cite{AfanasjewEtAl2008a} and integral representations for the error based on Cauchy's integral formula for analytic functions~\cite{FrommerGuettelSchweitzer2014a} 
were developed \corrected{as an alternative}.
These representations significantly improve numerical stability and robustness compared to earlier representations.

The following result from~\cite{FrommerGuettelSchweitzer2014a} is valid for \emph{Stieltjes functions}, i.e., functions of the form
\corrected{
\begin{equation*}
    F(s) = \int_0^\infty \frac{1}{t+s} \dif{\mu(t)},
\end{equation*}
where $\mu$ is a nonnegative measure on the positive real axis such that $\int_0^\infty \frac{1}{t+1} \dif{\mu(t)} < \infty$. For many practically relevant Stieltjes functions, one has $\dif\mu(t) = \rho(t)\dif t$, i.e.,
\begin{equation}\label{eq:stieltjes_function}
    F(s) = \int_0^\infty \frac{\rho(t)}{t+s} \dif{t},
\end{equation}
where $\rho(t) 
\geq 0$ is a piecewise continuous function on $(0, \infty)$. 
}
Important examples of Stieltjes functions are $F(s) = s^{-\alpha}, \alpha \in (0, 1)$ and $F(s) = \log(1+s)/s$.

\begin{theorem}[adapted from Theorem~3.4 and Corollary~3.5 in~\cite{FrommerGuettelSchweitzer2014a}]\label{thm:stieltjes_error}
Let $F$ be a function of the form~\eqref{eq:stieltjes_function}, assume that $\spec(A) \cap (-\infty,0] = \emptyset$, denote by $f_m^{(k)}$ the restarted Arnoldi approximation for $F(A)b$ from the $k$th Arnoldi cycle and let $H_m^{(j)}, V_m^{(j)}, j = 1,\dots,k$ be the Hessenberg matrix and orthonormal basis from the $j$th Arnoldi cycle. 
Let $\psi_m^{(j)}(t) = \Tra{e}_m (H_m^{(j)} + tI)^{-1}e_1$.
Then\footnote{Note that the result as given in~\cite{FrommerGuettelSchweitzer2014a} has an erroneous factor of $-1$.}
\begin{eqnarray*}\label{eq:error_function_stieltjes}
    F(A)b-f_m^{(k)} &=& (-1)^{k} \big( \prod_{j=1}^k h_{m+1,m}^{(j)} \big) \norm{b} \int_0^\infty \rho(t) \big(\prod_{j=1}^k\psi_m^{(j)}(t) \big) (A+tI)^{-1} v_{m+1}^{(k)} \dif{t} \\
     &\eqqcolon& F^{(k+1)}(A)v_{m+1}^{(k)}.
\end{eqnarray*}
\end{theorem}

\corrected{We will present an alternative proof to the one from \cite{FrommerGuettelSchweitzer2014a} for \cref{thm:stieltjes_error} later in \cref{cor:error_laplace_to_stieltjes}.}
Based on Theorem~\ref{thm:stieltjes_error}, a convergence analysis of the restarted Arnoldi process for Stieltjes functions was presented in \cite{FrommerGuettelSchweitzer2014b}. From a computational perspective, the error representation given in \cref{thm:stieltjes_error} has the shortcoming that it cannot be evaluated by a closed formula. In~\cite{FrommerGuettelSchweitzer2014a}, it is proposed to approximately evaluate $F^{(k)}$ by adaptive numerical quadrature. The crucial observation in this context is that one only needs to evaluate $F^{(k)}$ at the (small) Hessenberg matrix $H_m^{(k+1)}$, and not at $A$ itself. Thus, when implemented carefully, the computational cost of this quadrature rule is negligible compared to matrix-vector products with $A$ and orthogonalization costs. For an overview of several different, suitable quadrature rules tailored for specific practically relevant Stieltjes functions, we refer to~\cite[Section~4]{FrommerGuettelSchweitzer2014a}.

\begin{remark}\label{rem:rt_restarting}
There is another, conceptually different way of performing restarts in the Arnoldi method, the so-called ``residual-time restarting'' developed in~\cite{BotchevKnizhnerman2020,BotchevKnizhnermanTyrtyshnikov2020}, based on~\cite{BotchevGrimmHochbruck2013}. This approach is limited to functions like the exponential and $\varphi$-functions since it exploits the connection to an underlying ordinary differential equation. Instead of trying to approximate an \emph{error function}, as outlined above, it is based on a \emph{time-stepping approach}, where each restart cycle propagates the iterate from the previous cycle forward in time, until the desired point is reached.
\end{remark}

\subsection{The Laplace transform}\label{subsec:laplace_transform}
The function given by the integral
\begin{equation}\label{eq:laplace_transform}
    \mathcal{L}_t\{f(t)\}(s) = \int_0^\infty f(t) \exp(-st) \dif{t}
\end{equation}
is called the \emph{Laplace transform of $f$}. Whenever the integration variable is clear, we will simply write $\Lp{f}(s)$. 
In this section we summarize the most important properties of Laplace transforms needed in this paper. 

The set of all values $s\in\complns$ for which $\Lp{f}(s)$ converges (absolutely) is called the \emph{region of (absolute) convergence}. It generally has the shape of a half plane:
\begin{theorem}[{\hspace{1sp}\cite[Theorem 3.1]{Doetsch}}] \label{thm:laplace_absolute_convergence}
    If \cref{eq:laplace_transform} converges absolutely at $s_0$, then it converges absolutely in the closed right half-plane $\real{s} \geq \real{s_0}$.
\end{theorem}

The smallest value $\alpha$ such that $\Lp{f}(s)$ converges absolutely for  $\real{s} > \alpha$ is called the \emph{abscissa of absolute convergence}. In the case of simple convergence, we obtain a (possibly open) right half-plane $\real{s} > \alpha_0$ with $\alpha_0 \leq \alpha$ being called the abscissa of convergence.  However, only considering absolute convergence is not a restriction in the sense of the following theorem:
\begin{theorem}[{\hspace{1sp}\cite[Theorem 3.4]{Doetsch}}]
    If \cref{eq:laplace_transform} converges for $s_0$, then it converges in the open half-plane $\real{s} > \real{s_0}$, where it can be expressed by the absolutely converging Laplace transform
    \begin{equation*}
        \Lp{f}(s) = (s-s_0) \Lp{\phi}(s-s_0)
    \end{equation*}
    with
    \begin{equation*}
        \phi(t) = \int_0^t \exp(-s_0\tau) f(\tau) \dif{t}.
    \end{equation*}
\end{theorem}

Within its region of convergence, a Laplace transform always represents an analytic function.
\begin{theorem}[{\hspace{1sp}\cite[Theorem 6.1]{Doetsch}}]\label{thm:laplace_analytic}
    Let $\Lp{f}(s)$ converge for $\real{s} > \alpha_0$. Then all its derivatives exist for $\real{s} > \alpha_0$, and they are Laplace transforms, too,
    \begin{equation*}
        \left(\frac{\dif}{\dif{s}}\right)^n \Lp{f}(s) = (-1)^n \Lp{t^n f(t)}(s), \quad n\in\mathbb{N}.
    \end{equation*}
\end{theorem}

Suppose we are given a function $F(s)$ and want to represent it as a Laplace transform. By \cref{thm:laplace_analytic} a necessary condition on $F$, such that $\Lp{f}(s) = F(s)$ exists, is that $F(s)$ itself is analytic in a region $\real{s} > \alpha_0$. Some publications have examined more precise conditions, see, e.g., \cite{Doetsch,Widder,Wilson86,Cooper64} and the Paley-Wiener theorem, e.g., in \cite{Rudin}.

A sufficient condition for $F$ to be a Laplace transform is that $F(s)$ is a Stieltjes function, see, e.g., \cite[Theorem~4a]{Widder}. This can easily be verified by exploiting that
\begin{equation*}
    \frac{1}{t+s} = \mathcal{L}_\tau\{\exp(-t\tau)\}(s),\quad \real{s} > -t,
\end{equation*}
so that we can rewrite \cref{eq:stieltjes_function} as
\begin{equation*}
    \int_0^\infty \frac{\rho(t)}{t+s} \dif{t} = \Lp{\Lp{\rho}}(s).
\end{equation*}
Since this condition is not necessary, see, e.g.,~\cite{Berg2008}, we know that the class of Stieltjes functions is a subclass of the class of all functions that allow a Laplace transform representation.

Krylov methods for $F(A)b$ where $F(z)$ is a Laplace transform (or, more generally, a \emph{Laplace--Stieltjes function}) have been considered before. E.g., in~\cite{BenziSimoncini2017}, a ``tensorized'' Krylov method for efficiently approximating functions of certain Kronecker-structured matrices is proposed, while \cite{MasseiRobol2021} discusses pole selection for the rational Krylov method for Laplace transforms. Remark 1 in \cite{Druskin2008} mentions that for Hermitian matrices the error $\varepsilon_m$ of the Lanczos approximation decreases strictly monotonically if $F(z) = \Lp{f}(z)$ with real nonnegative $f$. This result was later generalized to extended Krylov subspace methods in~\cite{Schweitzer2016b}. However, no restart approach for general Laplace transforms has been developed so far.

\section{Restarts for Laplace transforms: Theory}\label{sec:laplace_theory}
This section contains our main theoretical result, \cref{thm:error_laplace}. It is built on a known representation of $\varepsilon_m$ for $F(z) = \exp(-tz)$ that we restate in \cref{le:error_exp}. We conclude the section by emphasizing the connection to the earlier work on Stieltjes functions in \cref{cor:error_laplace_to_stieltjes}.

Suppose we are interested in the action of the exponential function,
\begin{equation*}
    y(t) = \exp(-tA)b.
\end{equation*}
The Arnoldi approximation in this case is given by
\begin{equation*}
    y_m(t) = \norm{b} V_m\exp(-tH_m)e_1.
\end{equation*}
We explicitly include the dependency on $t$, here, since we need it later. The following lemma gives an expression for the error $\varepsilon_m(t)$. This result has already been used, e.g., to obtain error bounds in \corrected{\cite[eq.~(32)]{Druskin98} for Hermitian $A$ and in \cite[Theorem 3.1]{Wang17} for non-Hermitian $A$. We repeat the proof here for convenience.} 
\begin{lemma}\label{le:error_exp}
    The error $\varepsilon_m(t) = y(t) - y_m(t)$ can be written as 
    \begin{equation*}
        \varepsilon_m(t) = -h_{m+1,m} \norm{b} \int_0^t \exp((\tau - t)A) v_{m+1} g(\tau) \dif{\tau}
    \end{equation*}
    where
    \begin{equation}\label{eq:emexpHe1}
        g(\tau) = \Tra{e}_m \exp(-\tau H_m)e_1
    \end{equation}
    is the $(m,1)$ entry of $\exp(-\tau H_m)$.
\end{lemma}
\corrected{
\begin{proof}
    First note that
    \begin{equation*}
        y_m'(t) = - \norm{b} V_m H_m \exp(-tH_m) e_1 =  -Ay_m(t) + h_{m+1,m} \norm{b}\Tra{e}_m \exp(-tH_m)e_1 v_{m+1},
    \end{equation*}
    where in the second equality we have used the Arnoldi relation \cref{eq:arnoldi_relation}.
    The derivative of the error thus satisfies 
    \begin{equation*}
        \varepsilon_m'(t) = y'(t) - y_m'(t) = -A\varepsilon_m(t) - h_{m+1,m} \norm{b}g(t) v_{m+1}.
    \end{equation*}
    With $\varepsilon_m(0) = 0$, this is an initial-value problem. The assertion then follows, e.g., by the variation of constants formula.
\end{proof}
}

Since the matrix Laplace transform can be expressed as an integral involving the matrix exponential, we can use \cref{le:error_exp} to develop a new error representation. The result is given in the following theorem in which we use the field of values $W(A)$ of a square matrix $A$ defined as $W(A) = \{x^{\mathsf{H}}Ax: \|x\|_2 = 1\}$.
\begin{theorem}\label{thm:error_laplace}
    Let $F(s) = \Lp{f(t)}(s)$ be the Laplace transform of $f(t)$. If
    \begin{equation*}
        \nu = \min_{\mu\in W(A)}\real{\mu}
    \end{equation*}
    lies within the region of absolute convergence of $\Lp{f(t)}(s)$,
    then the error $\varepsilon_m$ of the Arnoldi approximation can be represented via a matrix Laplace transform
    \begin{equation*}
        \varepsilon_m = F(A)b - \norm{b}V_mF(H_m)e_1 = -h_{m+1,m} \norm{b} \Lp{\tilde{f}}(A) v_{m+1},
    \end{equation*}
    where
    \begin{equation*}
        \tilde{f}(t) = \int_0^\infty f(t+\tau) g(\tau) \dif{\tau}, \enspace   \text{with } 
        g(\tau) = \Tra{e}_m \exp(-\tau H_m)e_1 \text{ (see \cref{eq:emexpHe1})}.
    \end{equation*}
    Moreover, $\nu$ lies in the region of absolute convergence of $\Lp{\tilde{f}}(s)$.
\end{theorem}
\begin{proof}
    For simplicity and without loss of generality we assume $\norm{b} = 1$. Then
    \begin{equation*}
        \varepsilon_m = \int_0^\infty f(t) (\exp(-tA)b - V_m\exp(-tH_m)e_1) \dif{t}
    \end{equation*}
    and applying \cref{le:error_exp} gives
    \begin{equation*}
        \varepsilon_m = -h_{m+1,m} \int_0^\infty f(t) \int_0^t \exp((s-t)A) v_{m+1} g(s) \dif{s} \dif{t}.
    \end{equation*}
    With the transformation $\tau = t-s$ in the inner integral we obtain
    \begin{align}
        \varepsilon_m &= -h_{m+1,m} \int_0^\infty \int_0^t f(t) \exp(-\tau A) v_{m+1} g(t-\tau) \dif{\tau} \dif{t} \nonumber\\
                                &= -h_{m+1,m} \int_0^\infty \int_0^\infty f(t) \exp(-\tau A) v_{m+1} g(t-\tau) u(t-\tau) \dif{\tau} \dif{t} \label{eq:order_of_integration}
    \end{align}
    with $u(x)$ being the Heaviside step function,
    \begin{equation*}
        u(x) =  \begin{cases}
                    1 & \text{if } x\geq 0 \\
                    0 & \text{else}
                \end{cases}.
    \end{equation*}
    Assume for now that we can interchange the order of integration. Then 
    \begin{align*}
        \varepsilon_m &= -h_{m+1,m} \int_0^\infty \exp(-\tau A)v_{m+1} \int_0^\infty f(t) g(t-\tau) u(t-\tau) \dif{t}\dif{\tau} \\
                                &= -h_{m+1,m} \int_0^\infty \exp(-\tau A)v_{m+1} \int_\tau^\infty f(t) g(t-\tau) \dif{t}\dif{\tau},
    \end{align*}
    where
    \begin{equation*}
        \int_\tau^\infty f(t) g(t-\tau) \dif{t} = \int_0^\infty f(t+\tau) g(t) \dif{t} = \tilde{f}(\tau),
    \end{equation*}
    which is the assertion of the theorem. That the order of integration in \cref{eq:order_of_integration} can indeed be interchanged and that $\nu$ lies in the region of absolute convergence of $\Lp{\tilde{f}}$ is shown in \cref{sec:appendix}. 
\end{proof}
Since the representation for $\varepsilon_m$ uses again a matrix Laplace transform which converges absolutely for $\nu$, we can apply \cref{thm:error_laplace} for later restarts, too. We summarize this as a corollary.
\begin{corollary} \label{cor:error_laplace_k}
    Let the assumptions of \cref{thm:error_laplace} hold. Define
    \begin{equation*}
        f^{(j)}(t) = \int_0^\infty f^{(j-1)}(t+\tau)g^{(j-1)}(\tau) \dif{\tau}, \corrected{\quad j\geq 2}
    \end{equation*}
    with $f^{(1)} = f$ and $g^{(j)}(\tau) = \Tra{e}_m\exp(-\tau H_m^{(j)}) e_1$ \corrected{for $j\geq 1$}.
    Then the error $\varepsilon_m^{(k)}$ after $\corrected{k \geq 1}$ restart cycles satisfies
    \begin{equation*}
        \varepsilon_m^{(k)} = (-1)^k \big(\prod_{j=1}^k h^{(j)}_{m+1,m}\big) \norm{b} \Lp{f^{(k+1)}}(A) v^{(k)}_{m+1}.
    \end{equation*}
\end{corollary}

\begin{remark}
    Let $A$ be Hermitian and $F(z)=\Lp{f}(z)$ with real nonnegative $f=f^{(1)}$. Then  the error $\varepsilon_m^{(1)}$ decreases strictly monotonically with growing $m$; see the discussion at the end of \cref{sec:basics}. 
    In this case we also have that $\tilde{f} = f^{(2)}$ has constant sign.
    Thus, by induction, the error also decreases monotonically within every following 
    restart cycle.
\end{remark}

\Cref{thm:error_laplace} gives an error function representation for the restarted Arnoldi algorithm in terms of a Laplace transform. Before we move to
possible extensions in \cref{sec:extensions},
we now show that the error representation for Stieltjes functions can be considered a special case of \cref{cor:error_laplace_k}.

\begin{corollary}\label{cor:error_laplace_to_stieltjes}
    If $F(s)$ is a Stieltjes function, i.e.,
    \begin{equation*}
        F(s) = \Lp{f}(s) = \Lp{\Lp{\rho}}(s),
    \end{equation*}
    then the error function representation in \cref{cor:error_laplace_k} reduces to \cref{thm:stieltjes_error} in the sense that
    \begin{equation*}
        \Lp{f^{(k+1)}}(A) v^{(k)}_{m+1} = \int_0^\infty \rho(t) \big(\prod_{j=1}^k\psi_m^{(j)}(t) \big) (A+tI)^{-1} v_{m+1}^{(k)} \dif{t}.
    \end{equation*}
\end{corollary}
\begin{proof}
    Due to the fact that every Stieltjes function is a double Laplace transform, we can rewrite the statement of the corollary as
    \begin{equation*}
        \Lp{f^{(k+1)}}(A) v^{(k)}_{m+1} = \mathcal{L}^2\{\rho(t) \big(\prod_{j=1}^k\psi_m^{(j)}(t)\big)\}(A)v^{(k)}_{m+1},
    \end{equation*}
    and it is sufficient to show that
    \begin{equation*}
        f^{(k+1)} = \Lp{\rho(t) \big(\prod_{j=1}^k\psi_m^{(j)}(t)\big)}.
    \end{equation*}
    We prove this by induction over $k\geq 0$: The case $k=0$ is trivial, since by the corollary's hypothesis we have $f^{(1)} = \Lp{\rho}$. Now, assume that the above equation holds for $k-1$, i.e., $f^{(k)} = \Lp{\rho(t) \big(\prod_{j=1}^{k-1}\psi_m^{(j)}(t)\big)}$. 
    Then, by definition of $f^{(k+1)}$,
    \begin{align*}
        f^{(k+1)}(s) &= \int_s^\infty f^{(k)}(t) g^{(k)}(t-s) \dif{t} \\
            &= \Lp{ \mathcal{L}^{-1}\{f^{(k)}\} \Lp{g^{(k)}} }(s) = \Lp{\rho(t) \big(\prod_{j=1}^{k-1}\psi_m^{(j)}(t)\big) \Lp{g^{(k)}}},
    \end{align*}
    where for the second equality we used \cite[Theorem~2.1]{AlAhmad2020}. What remains to show is that $\Lp{g^{(k)}} = \psi_m^{(k)}$ holds, which follows as
    \begin{align*}
        \Lp{g^{(k)}}(t) = \Tra{e}_m \mathcal{L}_\tau\{\exp(-\tau H^{(k)}_m)\}(t) e_1 = \Tra{e}_m (tI + H^{(k)}_m)^{-1} e_1 = \psi^{(k)}_m(t).
    \end{align*}
\end{proof}

\section{Extensions}\label{sec:extensions} The approach developed for Laplace transforms can easily be ported to two further classes of functions as we shortly discuss now.

\subsection{Two-sided Laplace transforms}
\emph{Two-sided} Laplace transforms
\begin{equation*}
    \mathcal{B}\{f\}(s) = \int_{-\infty}^\infty f(t) \exp(-st) \dif{t} = \Lp{f(-t)}(-s) + \Lp{f(t)}(s)
\end{equation*}
can be interpreted as the generalization of the Fourier transform
\begin{equation*}
    \int_{-\infty}^\infty f(t) \exp(-i\omega t) \dif{t}
\end{equation*}
to complex arguments $\omega$.
The region of (absolute) convergence in the two-sided case is generally a vertical strip instead of a half-plane. Computationally, we can approximate $\mathcal{B}\{f\}(A)b$  using the Arnoldi approach for the two Laplace transforms $\Lp{f(-t)}(-A)$ and $\Lp{f(t)}(A)$. Note that the Krylov subspaces for both Laplace transforms are identical and need thus be built only once.

\subsection{Complete Bernstein functions} \label{subsec:bernstein}
A function $F: (0,\infty) \to [0,\infty)$ is called a \textit{Bernstein function} if $F\in C^\infty$ and its derivative $F'$ is completely monotone, i.e.,
\begin{equation*}
    \left(\frac{\dif}{\dif{s}}\right)^n (-1)^{n-1} F(s) \geq 0 \quad \text{ for $s > 0$ }, \enspace n\in\mathbb{N}.
\end{equation*}
See \cite{Berg2008,Schilling2012} for more information about Bernstein functions.

A Bernstein function $F(s)$ is called \textit{complete} if it admits the representation
\begin{equation}\label{eq:bernstein_represent}
    F(s) = c + as + \int_0^\infty (1-\exp(-st)) f(t) \dif{t}
\end{equation}
where $f(t)$ is completely monotone 
and $c, a\geq 0$.\footnote{In fact, every Bernstein function can be represented as $F(s) = c + as +\int_0^\infty (1-\exp(-st)) \dif{\mu}(t)$ where $\mu$ is a positive measure satisfying $\int_0^\infty \frac{t}{t+1} \dif{\mu}(t) < \infty$, see \cite[Theorem~3.2]{Schilling2012}. Thus, the measure $\mu$ of a complete Bernstein function has a completely monotone density $f$ with respect to the Lebesgue measure.}
If for a complete Bernstein function we compute $(cI + aA)b$ directly and use the Arnoldi approximation on $\mathcal{K}_m(A,b)$ for the integral in \cref{eq:bernstein_represent}, then the error of the resulting approximation $f_m$ is
\begin{align}
    F(A)b - f_m &= \int_0^\infty (I-\exp(-tA))b f(t) \dif{t} - \int_0^\infty V_m(I-\exp(-tH_m))e_1 f(t) \dif{t} \nonumber\\
        &= -\int_0^\infty (\exp(-tA)b - V_m\exp(-tH_m)e_1) f(t) \dif{t}, \label{eq:berstein_error}
\end{align}
which is the error for the Laplace transform $-\Lp{f}$ provided it exists.
Consequently, the error representation for Laplace transforms can also be used for complete Bernstein functions. Some care is required, though, regarding the regions of convergence, since at first sight it seems as if $\nu$ should lie in the
region of absolute convergence of $\Lp{f}$ with $f$ from \cref{eq:bernstein_represent}, and this might be quite restrictive. However, an inspection of the proof of \cref{thm:error_laplace} given in \cref{sec:appendix} shows that it is sufficient to have that $\nu$ lies in the region of absolute convergence of $\Lp{tf(t)}$. Since from  \cref{eq:bernstein_represent} we have
    \begin{equation*}
        F'(s) = a + \int_0^\infty t\exp(-st) f(t) \dif{t} = a + \Lp{tf(t)}(s),
    \end{equation*}
    this means that the error representation in \cref{thm:error_laplace} and \cref{cor:error_laplace_k} is also valid with converging integrals if $F$ is a complete Bernstein function for which $\nu$ lies in the region of absolute convergence of $F'$.

\section{Restarts for Laplace transforms: Implementation aspects}\label{sec:laplace_implementation}
The representation of the error function in \cref{thm:error_laplace} enables us to design a new restarted Arnoldi algorithm for Laplace transforms: In \cref{alg:generic_arnoldi}, we have to set
\begin{equation*}
     F^{(k)}(H_m^{(k)})e_1 = (-1)^{k-1} \big(\prod_{j=1}^{k-1} h_{m+1,m}^{(j)}\big) \norm{b} \Lp{f^{(k)}}(H_m^{(k)})e_1,
\end{equation*}
see \cref{cor:error_laplace_k}. 
The new algorithm is given in \cref{alg:restarted_laplace}, containing references to important aspects of implementation discussed in the remainder of this section.
\begin{algorithm}
\caption{Restarted Arnoldi method for $F(A)b$ where $F$ is a Laplace transform}\label{alg:restarted_laplace}
\begin{algorithmic}[1]
\State Compute the Arnoldi decomposition $AV_m^{(1)} = V_m^{(1)}H_m^{(1)} + h_{m+1,m}^{(1)}v_{m+1}^{(1)}\Tra{e}_m$ \Statex\hspace{\algorithmicindent}with respect to $A$ and $b$.
\State Set $f_m^{(1)} = \norm{b}V_m^{(1)}F(H_m^{(1)})e_1$.
\State Choose a quadrature rule with nodes $t_i$ and weights $ w_i$, $i=1,\dots,\ell$, \corrected{to evaluate \Statex\hspace{\algorithmicindent}the Laplace transform of $f$ at $H_m^{(1)}$ (see \cref{subsec:quadrature})}.
\For{$k = 2, 3, \dots$ until convergence}
    \State Compute the Arnoldi decomposition $AV_m^{(k)} = V_m^{(k)}H_m^{(k)} + h_{m+1,m}^{(k)}v_{m+1}^{(k)}\Tra{e}_m$ \Statex\hspace{\algorithmicindent}\hspace{\algorithmicindent}with respect to $A$ and $v_{m+1}^{(k-1)}$.
    \corrected{\State Choose a quadrature rule with nodes $t_i$ and weights $ w_i$, $i=1,\dots,\ell$, to evaluate \Statex\hspace{\algorithmicindent}\hspace{\algorithmicindent}the Laplace transform of $f^{(k)}$ at $H_m^{(k)}$ (see \cref{subsec:quadrature}).}
    \If{$k=2$}
        \State Define $s^{(k-1)} = f^{(k-1)}$.
    \Else
        \State Construct a spline $s^{(k-1)}$ that interpolates $f^{(k-1)}$.
    \EndIf
    \State Approximate $(f^{(k)}(t_i))_{i=1,\dots,\ell}$ as in \cref{eq:spline_approx} using $s^{(k-1)}$ .
    \State Compute $d_m^{(k-1)} = (-1)^{k-1} \big(\prod_{j=1}^{k-1} h_{m+1,m}^{(j)}\big)\norm{b}V_m^{(k)}\Lp{f^{(k)}}(H_m^{(k)})e_1$ by 
    \Statex\hspace{\algorithmicindent}\hspace{\algorithmicindent}numerical quadrature  (see \cref{eq:quadrature_laplace}).
    \State Set $f_m^{(k)} = f_m^{(k-1)} + d_m^{(k-1)}$.
\EndFor
\end{algorithmic}
\end{algorithm}

\subsection{Quadrature}\label{subsec:quadrature}
To evaluate the error function, we apply a quadrature rule
\begin{equation} \label{eq:quadrature_laplace}
    \Lp{f^{(k)}}(H_m^{(k)})e_1 \approx \sum_{i=1}^\ell w_i f^{(k)}(t_i) \exp(-t_i H_m^{(k)})e_1.
\end{equation}
Because of the term involving the exponential function, we expect that the main contributions to the integral come from the interval containing the smaller real parts of the eigenvalues 
of $H_m^{(k)}$. 
Consequently, we determine a quadrature rule based on the smallest real part of the spectrum of $H_m^{(1)}$: We set
\begin{equation*}
     \nu = \min_{x\in\spec(H_m^{(1)})} \real{x}
\end{equation*}
and determine the nodes $t_i$ and the weights $w_i$ of the quadrature rule in \cref{eq:quadrature_laplace} in the same way MATLAB's routine \texttt{integral} would evaluate the integral
\begin{equation*}
     \Lp{f^{(1)}}(\nu) = \int_0^\infty f^{(1)}(t) \exp(-t\nu)\dif{t}
\end{equation*}
with relative and absolute target accuracy $\varepsilon_\text{q}$. This means that after the transformation $x = \sqrt{t}/(1+\sqrt{t})$, an adaptive Gauss-Kronrod scheme is applied to subintervals of the new integration interval $[0, 1)$; see \cite{Shampine2008} for more details. The nodes $t_i$ and weights $w_i$ obtained this way make up the quadrature rule \cref{eq:quadrature_laplace} for the integral $\Lp{f^{(k)}}(H_m^{(k)})$; if necessary, after each restart, the rule can be computed anew and thus adapted to the new error representation. To simplify notation we do not explicitly denote this dependency of $\ell$, $t_i$ and $w_i$ on $k$ in \cref{alg:restarted_laplace} and elsewhere. \corrected{In principle, one might aim at developing and applying more advanced quadrature procedures to our specific situation, but the method just outlined was sufficient in all our experiments.} 

\subsection{Spline interpolation}\label{subsec:splines}
Since $f^{(k)}(t) = \Tra{e}_k\mathcal{L}_\tau\{f^{(k-1)}(t+\tau)\}(H_m^{(k-1)})e_1$ can be interpreted as another Laplace transform evaluated at the Hessenberg matrix $H_m^{(k)}$, we can use the quadrature rule from above to evaluate $f^{(k)}$, too. However, $f^{(k)}$ is a recursive integral. While it is possible to apply a series of quadrature rules to the expanded expression
\begin{equation*}
    f^{(k)}(t) = \int_0^\infty \dots \int_0^\infty  f(t + \sum_{i=1}^{k-1} \tau_i) \prod_{i=1}^{k-1} g^{(i)}(\tau_i) \; \dif{\tau_1}\dots\dif{\tau_{k-1}},
\end{equation*}
this approach is prohibitively expensive except for very small $k$. Instead of the quadrature rule 
\begin{equation*}
    f^{(k)}(t) \approx \sum_{i=1}^\ell w_i f^{(k-1)}(t+t_i)g^{(k-1)}(t_i), 
\end{equation*}
we propose to use
\begin{equation} \label{eq:spline_approx}
    f^{(k)}(t) \approx \sum_{i=1}^\ell w_i s^{(k-1)}(t+t_i)g^{(k-1)}(t_i),
\end{equation}
where $(t_i,w_i)_{i=1,\dots,\ell}$ are the quadrature nodes and weights of \cref{subsec:quadrature} and $s^{(k-1)}$ is a cubic spline interpolating $f^{(k-1)}$ at points \corrected{$x_1 < x_2 < \dots < x_q$} whose selection we will discuss below.

Recall that an interpolating cubic spline is a function of the form
\begin{equation*}
    s(x) =
    \begin{cases}
        p_1(x) & \phantom{x_1 \leq {}} x \leq x_2 \\
        p_2(x) & x_2 \leq x \leq x_3 \\
        &\vdots \\
        p_{q-1}(x) & x_{q-1} \leq x 
    \end{cases},
\end{equation*}
where each $p_i$ is a cubic polynomial. \corrected{Note that we included extrapolation for $x< x_1$ and $x> x_q$ in the above formula.} If $s(x)$ interpolates $f(x)$ at the $q$ points $x_i$, \corrected{$i=1,\ldots,q$} then we have $s(x_i) = f(x_i)$ and $s$ satisfies the smoothness conditions
\corrected{
\begin{gather*}
    p_{j-1}(x_{j}) = p_{j}(x_{j}), \\
    p'_{j-1}(x_{j}) = p'_{j}(x_{j}), \\
    p''_{j-1}(x_{j}) = p''_{j}(x_{j}),
\end{gather*}
with $j=2,\dots, q-1$.} Our implementation uses MATLAB's \texttt{spline} functionality which applies ``not-a-knot'' boundary conditions (meaning that $p_1=p_2$ and $p_{q-2}=p_{q-1}$). For a detailed treatment of spline interpolation, see, e.g., \cite{deBoor_splines}.

We now discuss how to choose the interpolation nodes. In cycle $k-1$, we have computed
\begin{equation*}
    \Lp{f^{(k-1)}}(H_m^{(k-1)})e_1 \approx \sum_{i=1}^\ell w_i f^{(k-1)}(t_i) \exp(-t_i H_m^{(k-1)})e_1,
\end{equation*}
see \cref{subsec:quadrature}.
Consequently, the values $f^{(k-1)}(t_i)$ are already known in the next cycle $k$ and can be used as interpolation points for $s^{(k-1)}$. In this case, the quadrature nodes and the interpolation nodes coincide: $q = \ell$ and $x_i = t_i$.
However, we observed that sometimes additional interpolation points are needed to ensure overall convergence of $f_m^{(k)}$ to $F(A)b$. An adaptive way to obtain those is to add the points $(x_i + x_{i+1})/2$ to the set of interpolation nodes\footnote{We assume that $f^{(k-1)}$ is approximated again as in \cref{eq:spline_approx} for these additional evaluations.} and repeat the process until a suitable stopping criterion is fulfilled. If $d_r$ denotes the value of $d_m^{(k-1)}$ after refinement step $r$, a possible criterion is that 
\begin{equation}\label{eq:spl_refinement_cond}
    \norm{d_{r}-d_{r-1}} \leq \varepsilon_\text{s} \norm{f_m^{(k-1)}}
\end{equation} with $\varepsilon_\text{s}$ sufficiently small.
The number of necessary refinement steps can be predicted. For example, if $f^{(k-1)} \in C^4[x_1, x_n]$, then the error bound (see \cite[eq.~(1.8)]{Beatson})
\begin{equation*}
  \abs{f^{(k-1)}(t)-s^{(k-1)}(t)} \leq c \,\Delta{x_i}^2 \max_j \Delta{x_j}^2 \, \|\corrected{\frac{\dif^4}{\dif{t}^4} f^{(k-1)}}\|_\infty \enspace \text{ for }    x_i\leq t \leq x_{i+1}
\end{equation*}
is reduced by the factor $16$ in each step. If the number of interpolation nodes needed is $\mathcal{O}(\ell^2)$, it is more advantageous to use $t_i+t_j$ with $i=1,\dots,\ell$, $j=i,\dots,\ell$: In that case, we have
\begin{equation*}
    \sum_{i=1}^\ell w_i f^{(k-1)}(t+t_i)g^{(k-1)}(t_i) = \sum_{i=1}^\ell w_i s^{(k-1)}(t+t_i)g^{(k-1)}(t_i)
\end{equation*}
in \cref{eq:spline_approx} for each $t=t_i$ required by \cref{eq:quadrature_laplace}.

\subsection{Matrix exponential function}
In \eqref{eq:quadrature_laplace} and \eqref{eq:spline_approx} we need to compute $\exp(-t_iH_m^{(k)})e_1$ for several values of $t_i$. While MATLAB's \texttt{expm} is the state-of-the-art choice if the full matrix exponential is required, for our purposes we need an efficient approach to obtain the action of the matrix exponential on a vector, $\exp(-t_iH_m^{(k)})e_1$, and this for several values $t_i$. If $A$ is \corrected{Hermitian}, the Hessenberg matrices $H_m^{(k)}$ are \corrected{Hermitian}, too, and we use the eigendecomposition, i.e.,
\begin{equation*}
    \exp(-t_iH_m^{(k)})e_1 = (X_m^{(k)})^{\mathsf{H}}\exp(-t_i D_m^{(k)}) (X_m^{(k)}e_1)
\end{equation*}
where $X_m^{(k)}$ is the matrix of eigenvectors and $D_m^{(k)}$ is the diagonal matrix of eigenvalues of $H_m^{(k)}$. The eigendecomposition needs to be computed only once for all $t_i$. 
Since the eigendecomposition of a \corrected{non-Hermitian} matrix might involve non-trivial Jordan blocks or might be ill-conditioned, we take another approach for \corrected{non-Hermitian} $A$. We choose the MATLAB package \texttt{expmv}\footnote{available at \url{https://github.com/higham/expmv}} which implements an algorithm described in \cite{Al-Mohy_expmv}. It uses the scaling part of the scaling and squaring method and a truncated Taylor approximation together with some further  preprocessing steps. The method avoids matrix-matrix products, so that generally it will compute $\exp(-t_iH_m^{(k)})e_1$ faster than \texttt{expm}. For large $t_i$, however, the fact that the squaring part is missing in \texttt{expmv} makes \texttt{expm} faster. We modified \texttt{expmv} in our implementation to switch to \texttt{expm} in these cases.

\section{Numerical experiments}\label{sec:experiments}
We now present several numerical examples for the restarted Arnoldi method for matrix Laplace transforms.\footnote{source code available at \url{https://github.com/MaTso7/laplace_restarting}}
All examples were calculated in MATLAB R2021a on a laptop with Intel\textsuperscript{\textregistered}\ Core\texttrademark\ i7-8650U and $16\,\text{GB}$. We denote by $\mathtt{tol}$ the target relative error norm which was set to $10^{-7}$ in all our examples. To be on the safe side, the target accuracy for the quadrature rule $\varepsilon_\text{q}$ was chosen as $\varepsilon_\text{q} = 10^{-3} \mathtt{tol}$, although we observed that larger values up to $\mathtt{tol}$ usually worked as well. The splines $s^{(k)}$ were refined as described in \cref{subsec:splines} with $\varepsilon_\text{s} = \varepsilon_\text{q}$ in \cref{eq:spl_refinement_cond}. 
For comparisons, we also used the MATLAB package \texttt{funm\_quad}\footnote{available at \url{http://www.guettel.com/funm_quad}} \cite{FGS14b} that implements the algorithm of \cite{FrommerGuettelSchweitzer2014a}, where we chose the target accuracy for numerical integration to be $\mathtt{tol}$ \corrected{and, when applicable, the \emph{two-pass Lanczos method} which can be used as an alternative to restarting in the Hermitian case. This method runs the Lanczos process twice (once for assembling the tridiagonal matrix $H_m$, once for forming the approximation $f_m$ as linear combination of the basis vectors). This way, it is not necessary to store the full Krylov basis, at the price of roughly doubling the computational cost; see, e.g.,~\cite{FrommerSimoncini2008} for details.}

\corrected{%
\begin{remark}\label{rem:2pl}
When comparing the performance of restarted methods and of two-pass Lanczos, one has to keep in mind that just counting matrix-vector products does not give the full picture of the computational cost of two-pass Lanczos: In order to monitor convergence, one needs to evaluate some stopping criterion (e.g., the norm of the difference between iterates), which at iteration $j$ typically involves (at least) forming $F(H_j)$. If a large number of iterations is necessary, this induces a non-negligible additional cost, even if the stopping criterion is not checked after each iteration. In contrast, in a restarted method, one can compare the norm of iterates from subsequent cycles by only evaluating $F^{(k)}$ on matrices of the fixed size $m$. To make the methods comparable at least to some extent, we check for convergence every $m$ iterations, where $m$ is the restart length.
\end{remark}%
}

\subsection{Fractional negative power less than $-1$: $F(s) = s^{-3/2}$}\label{subsec:s32}
As a first example, consider
\begin{equation*}
    F(s) = s^{-3/2} = \frac{2}{\pi}\Lp{\sqrt{t}}(s).
\end{equation*}
This Laplace transform converges absolutely for $\real{s} > 0$. Since $G(s) = s^{-1/2}$ is a Stieltjes function and $F(s)=G(s)s^{-1}$, we can compute $F(A)b$ by first solving the linear system $c=A^{-1}b$ and then treating $G(A)(A^{-1}b)=G(A)c$ as a Stieltjes function. This gives us an established method to compare \cref{alg:restarted_laplace} against, even though the matrix function $F(A) = A^{-3/2}$ does not directly fit into any of the restart frameworks considered in the literature so far. For the action of the inverse, $A^{-1}b$, we use the MATLAB built-in implementations of CG when $A$ is symmetric positive definite and restarted GMRES otherwise.

As matrix $A$, we first consider the discretized 3D Laplacian on a cubic grid with constant step size and Dirichlet boundary conditions, i.e.,
\begin{equation*}
    A_\text{L} = A_1\oplus A_1 \oplus A_1 \in \mathbb{R}^{N^3\times N^3}.
\end{equation*}
Here $A_1 = \text{tridiag}(-1, 2, -1)\in \mathbb{R}^{N\times N}$ denotes the 1D Laplace operator and $\oplus$ the Kronecker sum \corrected{defined via the Kronecker product $\otimes$ as
\begin{equation*}
    M_1 \oplus M_2 = M_1 \otimes I_{m_2} + I_{m_1} \otimes M_2, \quad M_1 \in \complns^{m_1\times m_1}, M_2 \in \complns^{m_2 \times m_2}.
\end{equation*}
}
As a non-Hermitian example, we take the convection-diffusion operator on the unit cube  with the direction vector $[1, -1, 1]$ and a diffusion coefficient of $\epsilon = 10^{-3}$. Using first-order upwind discretization for the convection term then results in the matrix
\begin{equation*}
    A_\text{CD} = h^{-2} \epsilon A_\text{L} + h^{-1} (A_2\oplus \Tra{A}_2\oplus A_2) \in \mathbb{R}^{N^3\times N^3}
\end{equation*}
with $h=(N+1)^{-1}$ and $A_2 = \text{tridiag}(-1, 1, 0) \in \mathbb{R}^{N\times N}$. Note that both $A_\text{L}$ and $A_\text{CD}$ are positive definite.

In our experiment, we vary $N$ from $20$ to $100$, which gives matrix sizes between $8,\!000$ and $1,\!000,\!000$. We use a restart length of $m=50$ for $A_\text{L}$ and $m=20$ for $A_\text{CD}$. We stop the iteration at the end of the first restart cycle for which the relative error norm lies below the desired accuracy $\mathtt{tol}=10^{-7}$. To make sure that the solution of $A^{-1}b$ is sufficiently precise for this, we chose a target accuracy of $10^{-9}$ for the residual norm of CG and GMRES.
For both methods, \cref{fig:s32_matvec} reports the required number of matrix-vector products for the different matrix sizes.
We observe that the new restart approach requires significantly fewer matrix-vector products \corrected{than the Stieltjes function approach.} This is due to the fact that the ``first phase'', in which the CG or GMRES method is executed, is not necessary. To stress this effect further, we display the percentage of matrix-vector products that the first phase requires above each data point. We deduce that the restarted matrix function algorithms perform quite similarly and in fact, the Stieltjes function restarting method from~\cite{FrommerGuettelSchweitzer2014a} even requires a slightly smaller number of matrix-vector products. The advantage of the new method can thus almost completely be attributed to the fact that it directly approximates the desired quantity without needing to additionally solve a linear system first. \corrected{For the symmetric matrix $A_L$, when comparing with the two-pass Lanczos approach, we see that the new approach requires fewer matrix-vector multiplications except for the two largest values of $N$.}  

\begin{figure}\centering
\subfloat{
%
%
\begin{tikzpicture}
\begin{axis}[%
title=\titleforLaplace,
width=0.5\textwidth,
xmin=12,
xmax=108,
xlabel={$N$},
ymin=0,
ymax=1320,
ylabel={matvecs},
xmajorgrids,
ymajorgrids,
legend style={at={(0.03,0.97)}, anchor=north west, legend cell align=left, align=left, draw=white!15!black}
]
\addplot [laplace1]
  table[row sep=crcr]{%
20	100\\
30	200\\
40	300\\
50	350\\
60	400\\
70	550\\
80	700\\
90	800\\
100	950\\
};
\addlegendentry{Laplace}

\addplot [stieltjes1]
  table[row sep=crcr]{%
20	185\\
30	275\\
40	364\\
50	503\\
60	592\\
70	730\\
80	868\\
90	1006\\
100	1143\\
};
\addlegendentry{Stieltjes}

\node[right, align=left, rotate=35, font=\color{mycolor2}]
at (axis cs:15,225) {\scriptsize 0.46};
\node[right, align=left, rotate=35, font=\color{mycolor2}]
at (axis cs:25,315) {\scriptsize 0.45};
\node[right, align=left, rotate=35, font=\color{mycolor2}]
at (axis cs:35,404) {\scriptsize 0.45};
\node[right, align=left, rotate=35, font=\color{mycolor2}]
at (axis cs:45,543) {\scriptsize 0.4};
\node[right, align=left, rotate=35, font=\color{mycolor2}]
at (axis cs:55,632) {\scriptsize 0.41};
\node[right, align=left, rotate=35, font=\color{mycolor2}]
at (axis cs:65,770) {\scriptsize 0.38};
\node[right, align=left, rotate=35, font=\color{mycolor2}]
at (axis cs:75,908) {\scriptsize 0.37};
\node[right, align=left, rotate=35, font=\color{mycolor2}]
at (axis cs:85,1046) {\scriptsize 0.35};
\node[right, align=left, rotate=35, font=\color{mycolor2}]
at (axis cs:95,1183) {\scriptsize 0.34};

\addplot [twopass1]
  table[row sep=crcr]{%
20	200\\
30	300\\
40	400\\
50	400\\
60	500\\
70	600\\
80	700\\
90	700\\
100	800\\
};
\addlegendentry{Lanczos}

\end{axis}
\end{tikzpicture}%
} 
\subfloat{
%
%
\begin{tikzpicture}

\begin{axis}[%
title=\titleforConvdiff,
width=0.5\textwidth,
xmin=12,
xmax=108,
xlabel={$N$},
ymin=0,
ymax=950,
ylabel={\phantom{matvecs}},
xmajorgrids,
ymajorgrids,
legend style={at={(0.03,0.97)}, anchor=north west, legend cell align=left, align=left, draw=white!15!black}
]
\addplot [laplace1]
  table[row sep=crcr]{%
20	120\\
30	160\\
40	180\\
50	220\\
60	240\\
70	260\\
80	320\\
90	340\\
100	400\\
};

\addplot [stieltjes1]
  table[row sep=crcr]{%
20	237\\
30	327\\
40	427\\
50	449\\
60	518\\
70	585\\
80	657\\
90	715\\
100	799\\
};

\node[right, align=left, rotate=35, font=\color{mycolor2}]
at (axis cs:15,277) {\scriptsize 0.58};
\node[right, align=left, rotate=35, font=\color{mycolor2}]
at (axis cs:25,367) {\scriptsize 0.57};
\node[right, align=left, rotate=35, font=\color{mycolor2}]
at (axis cs:35,467) {\scriptsize 0.58};
\node[right, align=left, rotate=35, font=\color{mycolor2}]
at (axis cs:45,489) {\scriptsize 0.55};
\node[right, align=left, rotate=35, font=\color{mycolor2}]
at (axis cs:55,558) {\scriptsize 0.58};
\node[right, align=left, rotate=35, font=\color{mycolor2}]
at (axis cs:65,625) {\scriptsize 0.56};
\node[right, align=left, rotate=35, font=\color{mycolor2}]
at (axis cs:75,697) {\scriptsize 0.54};
\node[right, align=left, rotate=35, font=\color{mycolor2}]
at (axis cs:85,755) {\scriptsize 0.55};
\node[right, align=left, rotate=35, font=\color{mycolor2}]
at (axis cs:95,839) {\scriptsize 0.55};
\end{axis}
\end{tikzpicture}%
}
\caption{Number of matrix-vector products for approximating $A^{-3/2}b$. ``Laplace'' denotes \cref{alg:restarted_laplace}, ``Stieltjes'' is the combination of \texttt{funm\_quad} with CG (left) or GMRES (right). The small numbers above the dashed line indicate which fraction of the overall number of matrix-vector products was computed in the first phase, i.e., in the CG or GMRES method. \corrected{``Lanczos'' denotes the two-pass Lanczos method.} The restart length is $m$.}
\label{fig:s32_matvec}
\end{figure}
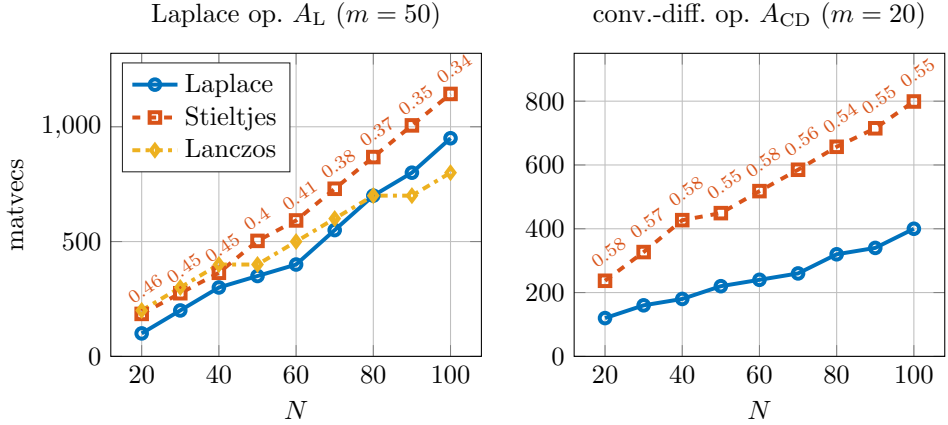

To illustrate that the reduction in matrix-vector products does not come with a decrease in accuracy, we also depict the final relative accuracies in \cref{fig:s32_matvec} in \cref{fig:s32_acc} as well as an exemplary convergence curve (for the largest problem instance $N=100$) in \cref{fig:s32_conv}. The convergence curve only covers the second phase of $A^{-1/2}(A^{-1}b)$, i.e., the approximation of the matrix function $A^{-1/2}$ applied to $c=A^{-1}b$.
We observe that the error reduction rate of both restarted methods is very similar, but that \texttt{funm\_quad} starts at a lower initial error norm so that it requires fewer restart cycles to reach the target accuracy. This is probably related to the different starting vectors $b$ and $A^{-1}b$ in both methods. \corrected{For $A_L$, non-restarted two-pass Lanczos exhibits superlinear convergence which would make this method require the least number of matrix-vector multiplications for accuracies beyond $10^{-7}$.} 

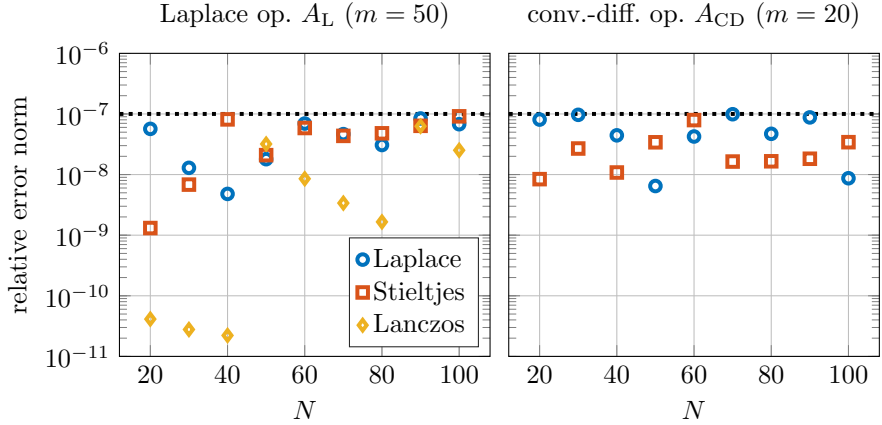
\begin{figure}\centering
\subfloat{
%
%
\begin{tikzpicture}

\begin{axis}[%
title=\titleforLaplace,
width=.5\textwidth,
xmin=12,
xmax=108,
xlabel={$N$},
ymode=log,
ymin=1e-11,
ymax=1e-06,
ylabel={relative error norm},
xmajorgrids,
ymajorgrids,
legend style={at={(0.97,0.03)}, anchor=south east, legend cell align=left, align=left, draw=white!15!black}
]
\addplot [laplace2]
  table[row sep=crcr]{%
20	5.67278170028882e-08 \\
30	1.29044574900336e-08 \\
40	4.79345695356431e-09 \\
50	1.79332907929683e-08 \\
60	6.94318430885119e-08 \\
70	4.64694390112736e-08 \\
80	3.07211126014486e-08 \\
90	8.41688603866269e-08 \\
100	6.77211509617863e-08 \\
};
\addlegendentry{Laplace}

\addplot [stieltjes2]
  table[row sep=crcr]{%
20	1.3110157501429e-09\\
30	6.82621613110315e-09\\
40	8.13413306382206e-08\\
50	2.09136292367843e-08\\
60	5.8619928893705e-08\\
70	4.34677315295056e-08\\
80	4.79175860996172e-08\\
90	6.36898322314859e-08\\
100	9.13309290071599e-08\\
};
\addlegendentry{Stieltjes}

\addplot [twopass2]
  table[row sep=crcr]{%
20	4.12502284422633e-11\\
30	2.77984421067966e-11\\
40	2.21756369340692e-11\\
50	3.19204361125245e-08\\
60	8.53297061650863e-09\\
70	3.38385339676338e-09\\
80	1.65542679815225e-09\\
90	6.27126614829866e-08\\
100	2.51624628029023e-08\\
};
\addlegendentry{Lanczos}

\addplot [tolline]
  table[row sep=crcr]{%
12	1e-07\\
108	1e-07\\
};
\end{axis}
\end{tikzpicture}%
} 
\subfloat{
%
%
\begin{tikzpicture}

\begin{axis}[%
title=\titleforConvdiff,
width=.5\textwidth,
xmin=12,
xmax=108,
xlabel={$N$},
ymode=log,
ymin=1e-11,
ymax=1e-06,
yticklabels={,,},
xmajorgrids,
ymajorgrids,
legend style={at={(0.97,0.03)}, anchor=south east, legend cell align=left, align=left, draw=white!15!black}
]
\addplot [laplace2]
  table[row sep=crcr]{%
20	8.07422784463715e-08\\
30	9.66855187505493e-08\\
40	4.43660312736290e-08\\
50	6.46980070943653e-09\\
60	4.24364413640115e-08\\
70	9.89258097967154e-08\\
80	4.70005979921947e-08\\
90	8.77200743055222e-08\\
100	8.69382938435744e-09\\
};

\addplot [stieltjes2]
  table[row sep=crcr]{%
20	8.40038849846075e-09\\
30	2.68899354220783e-08\\
40	1.08106587275625e-08\\
50	3.40761023302324e-08\\
60	7.91914767643469e-08\\
70	1.63908245267037e-08\\
80	1.65532323843993e-08\\
90	1.82358115552658e-08\\
100	3.42104347094156e-08\\
};

\addplot [tolline]
  table[row sep=crcr]{%
12	1e-07\\
108	1e-07\\
};
\end{axis}
\end{tikzpicture}%
}
\caption{Accuracy at termination when approximating $A^{-3/2}b$. ``Laplace'' denotes \cref{alg:restarted_laplace}, ``Stieltjes'' is the combination of \texttt{funm\_quad} with CG (left) or GMRES (right).  \corrected{``Lanczos'' denotes the two-pass Lanczos method.} The restart length is $m$.}
\label{fig:s32_acc}
\end{figure}

\begin{figure}\centering
\subfloat{
%
%
\begin{tikzpicture}

\begin{axis}[%
title=\titleforLaplace,
width=.5\textwidth,
xmin=0,
xmax=20*50,
xlabel={matvecs},
ymode=log,
ymin=5e-09,
ymax=7,
yminorticks=true,
ylabel={relative error norm},
xmajorgrids,
ymajorgrids,
yminorgrids,
legend style={legend cell align=left, align=left, draw=white!15!black}
]
\addplot [laplace1]
  table[row sep=crcr, x expr=\thisrowno{0} * 50]{%
1	0.704073706323051\\
2	0.284076152150897\\
3	0.111947289463698\\
4	0.0484540481987515\\
5	0.0189637151176775\\
6	0.00813721578900452\\
7	0.00317559860624965\\
8	0.00136055365020194\\
9	0.000530201304687341\\
10	0.000226895910172165\\
11	8.83311787147725e-05\\
12	3.77822543683065e-05\\
13	1.47039138937676e-05\\
14	6.28914783927941e-06\\
15	2.44745316215646e-06\\
16	1.04678374436637e-06\\
17	4.07300627378632e-07\\
18	1.74145666479802e-07\\
19	6.77211509617863e-08\\
};
\addlegendentry{Laplace}

\addplot [stieltjes1]
  table[row sep=crcr, x expr=\thisrowno{0} * 50]{%
1	0.107429814056415\\
2	0.0302444002125087\\
3	0.00947772695351599\\
4	0.00328013180061802\\
5	0.00113560879645632\\
6	0.000417466937559997\\
7	0.000150777544070397\\
8	5.70761171433699e-05\\
9	2.1090216369746e-05\\
10	8.12217871199697e-06\\
11	3.04473594349645e-06\\
12	1.18627473531704e-06\\
13	4.49582974746843e-07\\
14	1.80392355396477e-07\\
15	9.13309290071599e-08\\
};
\addlegendentry{funm\_quad}

\addplot [twopass1]
  table[row sep=crcr, x expr=\thisrowno{0} * 50]{%
2	0.704073706323045\\
4	0.279149949178622\\
6	0.0534761804523625\\
8	0.00438426054267289\\
10	0.000307820900113818\\
12	1.26852110290129e-05\\
14	8.31166106125833e-07\\
16	2.51624628029023e-08\\
};
\addlegendentry{Lanczos}

\addplot [tolline]
  table[row sep=crcr, x expr=\thisrowno{0} * 50]{%
0  1e-07\\
20 1e-07\\
};
\end{axis}
\end{tikzpicture}%
} 
\subfloat{
%
%
\begin{tikzpicture}

\begin{axis}[%
title=\titleforConvdiff,
width=.5\textwidth,
xmin=0,
xmax=25*20,
xlabel={matvecs},
ymode=log,
ymin=5e-09,
ymax=7,
yminorticks=true,
yticklabels={,,},
xmajorgrids,
ymajorgrids,
yminorgrids,
legend style={legend cell align=left, align=left, draw=white!15!black}
]
\addplot [laplace1]
  table[row sep=crcr, , x expr=\thisrowno{0} * 20]{%
1 0.688635134749085\\
2 0.626854843135054\\
3 0.576650994425777\\
4 0.526762190415617\\
5 0.476117629763813\\
6 0.424141703185140\\
7 0.369203054105566\\
8 0.310443915033232\\
9 0.248002902129046\\
10  0.183050493704130\\
11  0.118501077235002\\
12  0.0623868170599399\\
13  0.0262807171485371\\
14  0.00749316993005576\\
15  0.00140157235713879\\
16  4.87086464380148e-05\\
17  5.83498508659273e-06\\
18  8.90397921040723e-07\\
19  1.08689772492615e-07\\
20  8.69382938435744e-09\\
};

\addplot [stieltjes1]
  table[row sep=crcr, x expr=\thisrowno{0} * 20]{%
1	0.0792311452498629\\
2	0.0599510340536015\\
3	0.0481938729329755\\
4	0.0388621279913721\\
5	0.030918527903387\\
6	0.0242163072041038\\
7	0.0187216161609343\\
8	0.0144001615216047\\
9	0.0110623546076092\\
10	0.00831973849322236\\
11	0.00573192100512827\\
12	0.00315897483475379\\
13	0.00122026864773615\\
14	0.000310855800377583\\
15	4.79135736393372e-05\\
16	2.28056202571125e-06\\
17	2.05114893235413e-07\\
18	3.42104347094156e-08\\
};

\addplot [tolline]
  table[row sep=crcr, x expr=\thisrowno{0} * 20]{%
0  1e-07\\
25 1e-07\\
};
\end{axis}
\end{tikzpicture}%
}
\caption{Convergence curves for approximating $A^{-3/2}b$ with \cref{alg:restarted_laplace} (``Laplace'') and for approximating $A^{-1/2}c$ where $c = A^{-1}b$, with \texttt{funm\_quad}.  \corrected{``Lanczos'' denotes the two-pass Lanczos method.} $N=100$. The restart length is $m$.}
\label{fig:s32_conv}
\end{figure}
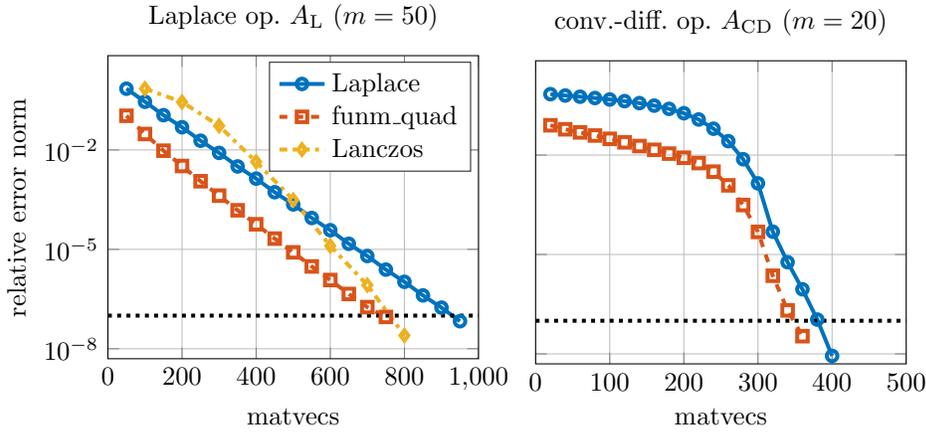

\Cref{fig:s32_time} illustrates that the lower number of matrix-vector products also translates into an advantage in run time, in particular for the larger problem instances. The overhead due to the spline construction and evaluation in \cref{alg:restarted_laplace} (and due to the matrix exponential for $A=A_\text{CD}$) is relatively large for the smaller problem sizes, so we do not observe a significant speed-up. For larger $N$, however, the speed-up is similar to what we would expect from \cref{fig:s32_matvec}. Moreover, we see linear dependence on the matrix size $N^3$, i.e., the overhead is negligible for larger matrices, the cost of the methods being  completely dominated by matrix-vector products. \corrected{Two-pass Lanczos does not yet exhibit the overhead as discussed in Remark~\ref{rem:2pl}, probably because the solution is found at a rather moderate Krylov subspace size of about 400.}

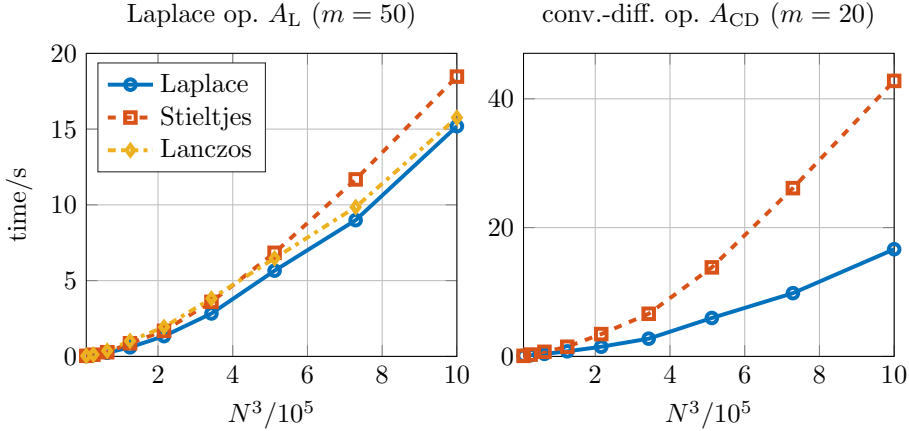
\begin{figure}\centering
\subfloat{
%
%
\begin{tikzpicture}
\begin{axis}[%
title=\titleforLaplace,
width=.5\textwidth,
xmin=7992,
xmax=1000008,
xlabel={$N^\text{3} / 10^{5}$},
ymin=0,
ymax=20,
ylabel={time/s},
xmajorgrids,
ymajorgrids,
scaled x ticks = false,
xtick={0,2e5,...,10e5},
xticklabels={0,2,...,10},
legend style={at={(0.03,0.97)}, anchor=north west, legend cell align=left, align=left, draw=white!15!black}
]
\addplot [laplace1]
  table[row sep=crcr]{%
8000	0.020275\\
27000	0.090037\\
64000	0.24035\\
125000	0.598949\\
216000	1.337009\\
343000	2.832405\\
512000	5.656938\\
729000	8.990075\\
1000000	15.189598\\
};
\addlegendentry{Laplace}

\addplot [stieltjes1]
  table[row sep=crcr]{%
8000	0.038029\\
27000	0.110386\\
64000	0.276226\\
125000	0.866981\\
216000	1.687112\\
343000	3.600581\\
512000	6.836653\\
729000	11.679088\\
1000000	18.463929\\
};
\addlegendentry{Stieltjes}

\addplot [twopass1]
  table[row sep=crcr]{%
8000	0.037465\\
27000	0.121189\\
64000	0.373204\\
125000	1.017581\\
216000	1.936739\\
343000	3.819076\\
512000	6.473644\\
729000	9.852824\\
1000000	15.756791\\
};
\addlegendentry{Lanczos}

\end{axis}
\end{tikzpicture}%
} 
\subfloat{
%
%
\begin{tikzpicture}

\begin{axis}[%
title=\titleforConvdiff,
width=.5\textwidth,
xmin=7992,
xmax=1000008,
xlabel={$N^\text{3}/10^5$},
ymin=0,
ymax=47.0642381,
xmajorgrids,
ymajorgrids,
scaled x ticks = false,
xtick={0,2e5,...,10e5},
xticklabels={0,2,...,10},
legend style={at={(0.03,0.97)}, anchor=north west, legend cell align=left, align=left, draw=white!15!black}
]
\addplot [laplace1]
  table[row sep=crcr]{%
8000	0.107253\\
27000	0.186035\\
64000	0.389328\\
125000	0.784314\\
216000	1.503119\\
343000	2.752375\\
512000	5.966028\\
729000	9.818296\\
1000000	16.641498\\
};

\addplot [stieltjes1]
  table[row sep=crcr]{%
8000	0.088814\\
27000	0.284351\\
64000	0.713247\\
125000	1.499026\\
216000	3.463649\\
343000	6.619234\\
512000	13.8243\\
729000	26.130229\\
1000000	42.785671\\
};

\end{axis}
\end{tikzpicture}%
}
\caption{Execution times when approximating $A^{-3/2}b$ for varying matrix sizes $N^3$. ``Laplace'' denotes \cref{alg:restarted_laplace}, ``Stieltjes'' is the combination of \texttt{funm\_quad} with CG (left) or GMRES (right). \corrected{``Lanczos'' denotes the two-pass Lanczos method.} The restart length is $m$.}
\label{fig:s32_time}
\end{figure}

The results hold for other restart lengths $m$, too. As a small demonstration, we report the run times for varying restart lengths for $N=100$ in \cref{fig:s32_time_rlen}. For $A_\text{L}$, we had to include an additional factor of $10^{-2}$ in the target accuracy of the numerical integration in both \cref{alg:restarted_laplace} and \texttt{funm\_quad} for $m=10$ and a factor of $10^{-1}$ in \texttt{funm\_quad} for $m=20,\, 30,\, 40$. While a larger $m$ can possibly improve convergence and thus decrease the run time, this was only observed for the Hermitian case $A_\text{L}$. The increasing cost of orthogonalization for the next Arnoldi vector in the non-Hermitian case $A_\text{CD}$ leads to an overall increasing execution time. In any case, \cref{fig:s32_time_rlen} shows that \cref{alg:restarted_laplace} behaves similarly to other restart algorithms when changing $m$.

\begin{figure}\centering
\subfloat{
%
%
\begin{tikzpicture}
\begin{axis}[%
title=\titleforLaplacewom,
width=.5\textwidth,
xmin=2,
xmax=108,
xlabel={$m$},
ymin=0,
ymax=60,
ylabel={time/s},
xmajorgrids,
ymajorgrids,
]
\addplot [laplace1]
  table[row sep=crcr]{%
10	49.5122640000000\\
20	27.4138180000000\\
30	21.6585290000000\\
40	17.1214830000000\\
50	14.8062270000000\\
60	13.0432370000000\\
70	12.0131500000000\\
80	11.1898190000000\\
90	9.75618400000000\\
100	10.8948530000000\\
};
\addlegendentry{Laplace}

\addplot [stieltjes1]
  table[row sep=crcr]{%
10	52.8322720000000\\
20	29.6948640000000\\
30	22.9511160000000\\
40	19.6930430000000\\
50	18.4766120000000\\
60	17.8809880000000\\
70	16.4907450000000\\
80	16.6809550000000\\
90	16.4002330000000\\
100	15.7849450000000\\
};
\addlegendentry{Stieltjes}

\addplot [twopass1]
  table[row sep=crcr]{%
10	21.605331\\
20	17.755568\\
30	16.237595\\
40	16.920052\\
50	15.671607\\
60	15.982951\\
70	15.908369\\
80	15.062025\\
90	16.712128\\
100	14.642329\\
};
\addlegendentry{Lanczos}

\end{axis}
\end{tikzpicture}%
} 
\subfloat{
%
%
\begin{tikzpicture}

\begin{axis}[%
title=\titleforConvdiffwom,
width=.5\textwidth,
xmin=2,
xmax=108,
xlabel={$m$},
ymin=0,
ymax=250,
xmajorgrids,
ymajorgrids,
]
\addplot [laplace1]
  table[row sep=crcr]{%
10  10.311031\\
20  16.021759\\
30  19.207483\\
40  26.440988\\
50  34.788093\\
60  42.676366\\
70  56.574208\\
80  56.353829\\
90  78.984958\\
100 82.904861\\
};

\addplot [stieltjes1]
  table[row sep=crcr]{%
10  27.62518\\
20  41.719782\\
30  55.726115\\
40  72.326795\\
50  93.428572\\
60  111.146174\\
70  151.481242\\
80  165.298014\\
90  164.77077\\
100 227.491125\\
};

\end{axis}
\end{tikzpicture}%
}
\caption{Execution times when approximating $A^{-3/2}b$ for varying restart lengths $m$ and $N=100$. ``Laplace'' denotes \cref{alg:restarted_laplace}, ``Stieltjes'' is the combination of \texttt{funm\_quad} with CG (left) or GMRES (right). \corrected{``Lanczos'' denotes the two-pass Lanczos method. Note that for two-pass Lanczos, $m$ specifies once in how many iterations the stopping criterion is checked.}}
\label{fig:s32_time_rlen}
\end{figure}
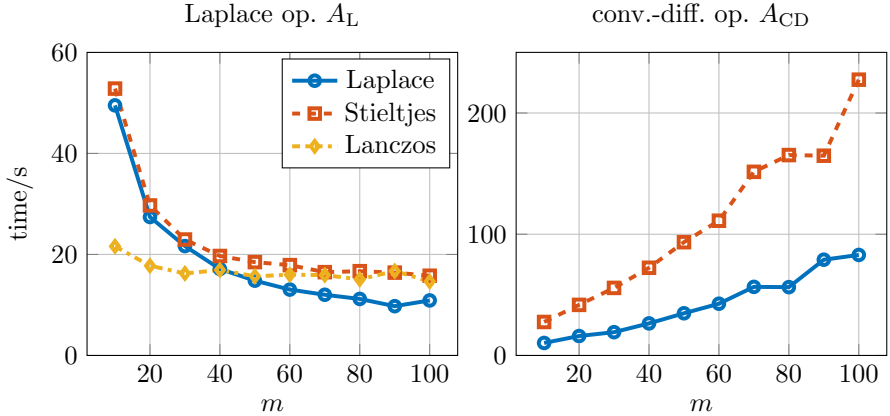

\subsection{Fractional diffusion processes on graphs: $F(s) = \exp(-\tau\sqrt{s})$} Given an (undirected) graph $G$, its symmetric positive semidefinite \emph{graph Laplacian} $L_G = D_G-A_G$, where $A_G$ is the adjacency matrix and $D_G$ the diagonal matrix containing the degrees of the nodes of $G$, the ODE system
\begin{equation}\label{eq:fractional_diffusion}
    \begin{cases}
    \frac{d}{d\tau} u(\tau) &= -L_G^{\alpha}u(\tau), \quad \tau \in (0, T], \\
    \phantom{\frac{d}{d\tau}} u(0) &= u_0
    \end{cases}
\end{equation}
(where $\alpha \in (0,1)$) models a fractional diffusion process on $G$; see, e.g.,~\cite{BenziSimunec2021}. Clearly, the solution of~\eqref{eq:fractional_diffusion} is given by
\begin{equation}\label{eq:fractional_diffusion_solution}
    u(t) = \exp(-\tau L_G^\alpha)u_0.
\end{equation}
It is well-known that the function $\exp(-\tau s^{\alpha})$ appearing in~\eqref{eq:fractional_diffusion_solution} is a Laplace--Stieltjes function, as it is the composition of a Bernstein function and a completely monotonic function. For $\alpha = 1/2$, we can give a closed form expression as a Laplace transform,
\begin{equation*}
    \exp(-\tau\sqrt{s}) = \frac{\tau}{2\sqrt{\pi}} \Lp{\frac{\exp(-\tau^2/(4t))}{t^{3/2}}}(s),
\end{equation*}
which converges absolutely for $\real{s}\geq 0$,
so that \cref{alg:restarted_laplace} can be applied. For invertible $A$, it is sometimes proposed to write 
\begin{equation} \label{product_with_A:eq}
\exp(-\tau\sqrt{A})u_0 = (h(A)A+I)u_0
\end{equation}
with the function $h(s) = (\exp(-\tau \sqrt{s})-1)/s$, which has the integral representation
\begin{equation}\label{eq:exp_sqrt_stieltjes}
h(s) = -\int_0^\infty \frac{1}{s+t}\frac{\sin(\tau\sqrt{t})}{\pi t} \dif{t};
\end{equation}
see, e.g.,~\cite{DruskinKnizhnerman1998}. While $h$ is not a Stieltjes function in the strict sense due to the oscillating generating function in~\eqref{eq:exp_sqrt_stieltjes}, the quadrature-based restarting technique from~\cite{FrommerGuettelSchweitzer2014a} is still applicable to it\footnote{By using the Abel--Dirichlet test~\cite[Theorem 11.23a]{Shilov1973}, one can easily check that also in this situation the integral representations of all occurring error functions are guaranteed to exist and have a finite value.}, thus giving an alternative way to approximate $\exp(-\tau\sqrt{A})$.
However, in the context of fractional diffusion on graphs, it is worth noting that $L_G$ is always a singular matrix so that $h(L_G)$ is not defined. As $L_Gu_0$ does not contain any contribution from the nullspace of $L_G$, it is possible to define the action $h(L_G)L_Gu_0$ by restriction to the orthogonal complement of the nullspace, but this makes some theoretical considerations more complicated. This holds in particular in the presence of round-off errors which may introduce (small) nullspace components during the Krylov iteration. Therefore, working directly with $\exp(-\tau\sqrt{s})$ is an attractive approach in this setting. 

\begin{table}
\centering
\caption{Number of nodes and edges of the largest connected components. The graphs were obtained from \cite{SuiteSparse}.}
\label{tab:graphs}
    \begin{tabular}{@{}lrr@{}}
    \toprule
        Name        & \multicolumn{1}{c}{nodes} & \multicolumn{1}{c}{edges} \\\midrule
        usroads-48  & $126\,146$    & $323\,900$\\
        loc-Gowalla & $196\,591$    & $1\,900\,654$\\
        dblp-2010   & $226\,413$    & $1\,432\,920$\\
        com-Amazon  & $334\,863$    & $1\,851\,744$\\
    \bottomrule
    \end{tabular}
\end{table}

\begin{figure}
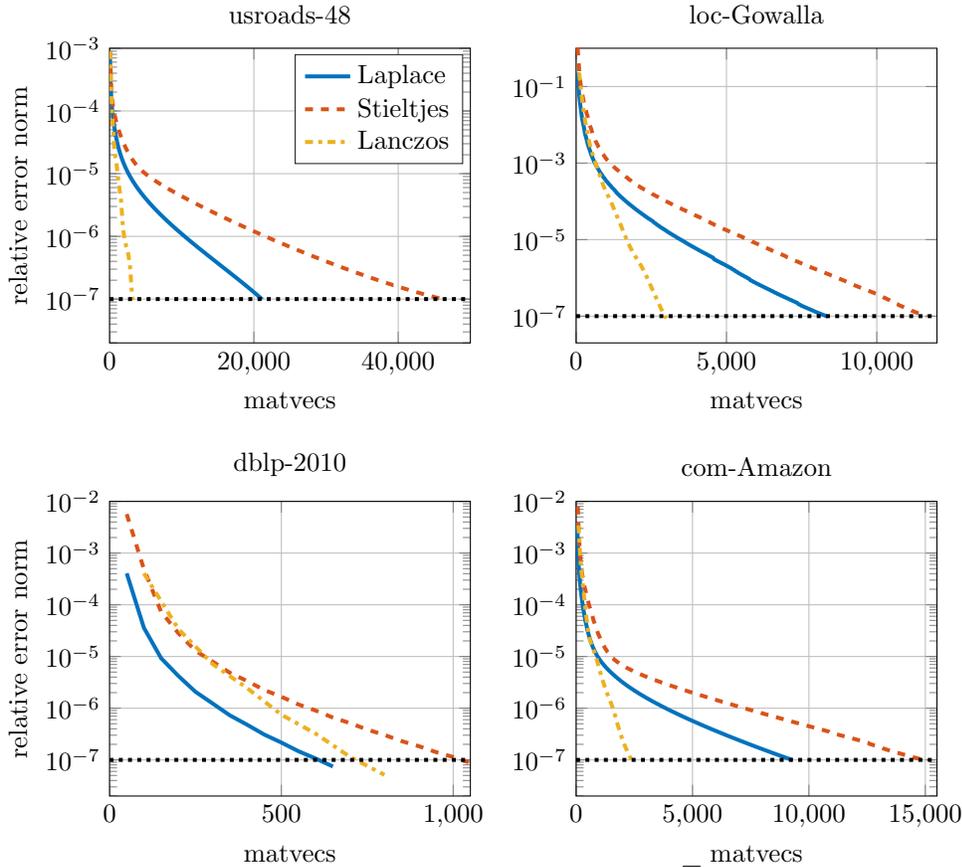
\centering
\begin{tikzpicture}
\begin{groupplot}[group style={group size=2 by 2, vertical sep=6em, horizontal sep=4em}, width=0.49\textwidth]
    \input{plots/fracdiff/usroads-48}
    \input{plots/fracdiff/loc-Gowalla}
    \input{plots/fracdiff/dblp-2010}
    \input{plots/fracdiff/com-Amazon}
\end{groupplot}
\end{tikzpicture}\vspace{-.5cm}
\caption{Convergence curves for approximating $\exp(-\sqrt{A})b$. ``Laplace'' denotes \cref{alg:restarted_laplace}, ``Stieltjes'' refers to the final error obtained by using \texttt{funm\_quad} for $h(A)(Ab)$ in $\exp(-\sqrt{A})b = (h(A)A+I)b$.  \corrected{``Lanczos'' denotes the two-pass Lanczos method.} The restart length is $m=50$.}
\label{fig:fracdiff_conv}
\end{figure}

We used four real-world graphs for our experiments. The adjacency matrices were obtained from the SuiteSparse Matrix Collection \cite{SuiteSparse}. For each graph, we extracted the largest connected component and only considered the corresponding Laplacian. \Cref{tab:graphs} contains the number of nodes and edges of the resulting graphs. For sake of simplicity, we chose $\tau=1$. The restart length is $m=50$. We present the convergence curves in \cref{fig:fracdiff_conv}. They show that in addition to the above argument, \cref{alg:restarted_laplace} yields higher accuracy \corrected{than the Stieltjes-type approach for the same number of matrix-vector products} for each considered graph. \corrected{Two-pass Lanczos requires the least number of matrix-vector multiplications except for the dblp-2010 graph.} 

Employing the ``corrected'' Arnoldi approximation introduced in~\cite{Saad1992}, we can modify the restarting method for Stieltjes functions $h$ to directly approximate $Ah(A)b$ without needing an explicit pre- or postmultiplication with $A$; see Corollary 3.6 and the discussion following it in~\cite{FrommerGuettelSchweitzer2014a}. We could apply this here, based on \eqref{product_with_A:eq}. For all graphs considered, however, this approach shows actually worse performance than the Stieltjes restart method applied to the starting vector $L_Gb$. It sometimes even stagnates or shows signs of instability, which we attribute to the fact that $L_G$ is singular. Thus premultiplying with $L_G$ is actually beneficial, removing unwanted nullspace contributions from $b$.

\subsection{Gamma function: $F(s) = \Gamma(s)$}
Recently, techniques to compute the matrix gamma function $\Gamma(A)$ were discussed in \cite{Cardoso2019,Miyajima20}. The gamma function can be represented as a two-sided Laplace transform,
\begin{equation*}
    \Gamma(s) = \int_0^\infty x^{s-1} \exp(-x) \dif{x} = \mathcal{B}\{\exp(-\exp(-t))\}(s),
\end{equation*}
which converges absolutely for $\real{s}>0$. To see the last equality, use $x^{s-1} = \exp((s-1)\log(x))$ and apply the transformation $t=-\log(x)$. 
We include this example to show that the approach presented in this paper 
also works for the two-sided case.
We compute $F(A)b$ by applying the method to the two Laplace transforms
\begin{equation*}
    F(A)b = \Lp{\exp(-\exp(-t))}(A)b + \Lp{\exp(-\exp(t))}(-A)b.
\end{equation*}
To determine the error, we implemented Algorithm 4.5 combined with 4.1 of \cite{Cardoso2019} to compute $\Gamma(A)$ directly. As this becomes prohibitively expensive for larger matrices, we use the 2D versions of the matrices in \cref{subsec:s32}, i.e.,
\begin{align*}
    A_\text{L}  &= A_1\oplus A_1 \in \mathbb{R}^{N^2\times N^2}, \\
    A_\text{CD} &= h^{-2} \epsilon A_\text{L} + h^{-1} (A_2\oplus \Tra{A}_2) \in \mathbb{R}^{N^2\times N^2},
\end{align*}
which are again positive definite. We let $N$ vary from $20$ to $120$.

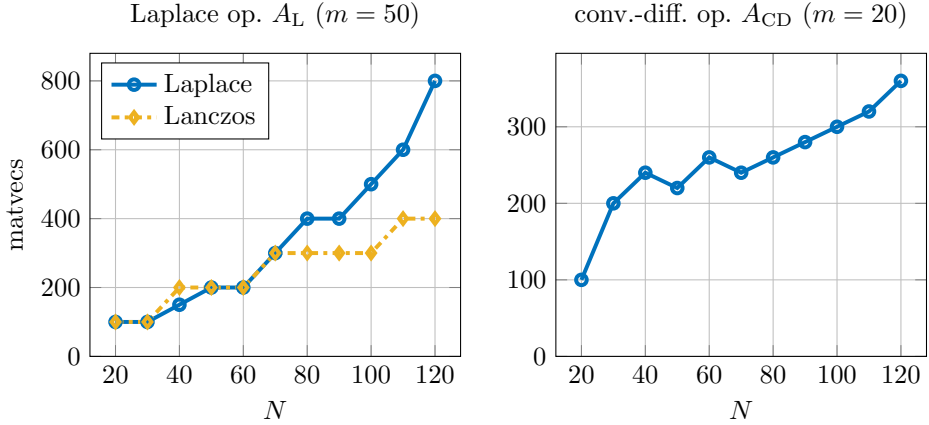
\begin{figure}\centering
\subfloat{
%
%
\begin{tikzpicture}

\begin{axis}[%
title=\titleforLaplace,
width=.5\textwidth,
xmin=12,
xmax=128,
xlabel={$N$},
ymin=0,
ymax=880,
ylabel={matvecs},
xmajorgrids,
ymajorgrids,
legend style={at={(0.03,0.97)}, anchor=north west, legend cell align=left, align=left}
]
\addplot [laplace1]
  table[row sep=crcr]{%
20	100\\
30	100\\
40	150\\
50	200\\
60	200\\
70	300\\
80	400\\
90	400\\
100	500\\
110	600\\
120	800\\
};
\addlegendentry{Laplace}

\addplot [twopass1]
  table[row sep=crcr]{%
20	100\\
30	100\\
40	200\\
50	200\\
60	200\\
70	300\\
80	300\\
90	300\\
100	300\\
110	400\\
120	400\\
};
\addlegendentry{Lanczos}

\end{axis}
\end{tikzpicture}%
} 
\subfloat{
%
%
\begin{tikzpicture}

\begin{axis}[%
title=\titleforConvdiff,
width=.5\textwidth,
xmin=12,
xmax=128,
xlabel={$N$},
ymin=0,
ymax=396,
ylabel={\phantom{matvecs}},
xmajorgrids,
ymajorgrids
]
\addplot [laplace1]
  table[row sep=crcr]{%
20	100\\
30	200\\
40	240\\
50	220\\
60	260\\
70	240\\
80	260\\
90	280\\
100	300\\
110	320\\
120	360\\
};
\end{axis}
\end{tikzpicture}%
}
\caption{Number of matrix-vector products for approximating $\Gamma(A)b$. The restart length is $m$. \corrected{``Laplace'' denotes \cref{alg:restarted_laplace}. ``Lanczos'' denotes the two-pass Lanczos method.}}
\label{fig:gamma_matvec}
\end{figure}
\begin{figure}\centering
\subfloat{
%
%
\begin{tikzpicture}

\begin{axis}[%
title=\titleforLaplace,
width=.5\textwidth,
xmin=12,
xmax=128,
xlabel={$N$},
ylabel={relative error norm},
ymode=log,
ymin=1e-14,
ymax=1e-06,
yminorticks=true,
xmajorgrids,
ymajorgrids,
legend style={at={(0.97,0.03)}, anchor=south east, legend cell align=left, align=left}
]
\addplot [laplace2]
  table[row sep=crcr]{%
20  5.09846447465670e-14\\
30  3.38989542847210e-14\\
40  1.10865450291262e-13\\
50  1.88917282350450e-11\\
60  1.31207772512194e-09\\
70  1.66924291819961e-09\\
80  9.78991743413287e-10\\
90  1.81256727001258e-08\\
100 1.83869615841184e-08\\
110 7.76640429861237e-08\\
120 2.48827040675529e-08\\
};
\addlegendentry{Laplace}

\addplot [twopass2]
  table[row sep=crcr]{%
20	4.52825270916375e-14\\
30	5.30105483140303e-09\\
40	2.58555730378476e-14\\
50	3.30340934525502e-12\\
60	3.5054007975051e-09\\
70	1.34187277455167e-13\\
80	5.92401180499846e-11\\
90	1.94409726133609e-09\\
100	3.71271948555034e-08\\
110	1.97727200179853e-10\\
120	1.34772127789919e-09\\
};
\addlegendentry{Lanczos}

\addplot [tolline]
  table[row sep=crcr]{%
12	1e-07\\
138	1e-07\\
};
\end{axis}
\end{tikzpicture}%
} 
\subfloat{
%
%
\begin{tikzpicture}

\begin{axis}[%
title=\titleforConvdiff,
width=.5\textwidth,
xmin=12,
xmax=128,
xlabel={$N$},
ymode=log,
ymin=1e-14,
ymax=1e-06,
yminorticks=true,
yticklabels={,,},
xmajorgrids,
ymajorgrids,
]
\addplot [laplace2]
  table[row sep=crcr]{%
20  2.32493172183863e-11\\
30  8.38703380283829e-10\\
40  1.81508361811180e-08\\
50  1.46072591023318e-09\\
60  1.87736649899729e-08\\
70  1.58321942491192e-09\\
80  8.96482026715075e-10\\
90  5.12712235726890e-09\\
100 4.15392014551614e-09\\
110 6.73026558070998e-09\\
120 1.22511330575328e-09\\
};
\addplot [tolline]
  table[row sep=crcr]{%
12	1e-07\\
128	1e-07\\
};
\end{axis}
\end{tikzpicture}%
}
\caption{Accuracy at termination when approximating $\Gamma(A)b$. \corrected{``Laplace'' denotes \cref{alg:restarted_laplace}. ``Lanczos'' denotes the two-pass Lanczos method.} The restart length is $m$.}
\label{fig:gamma_acc}
\end{figure}
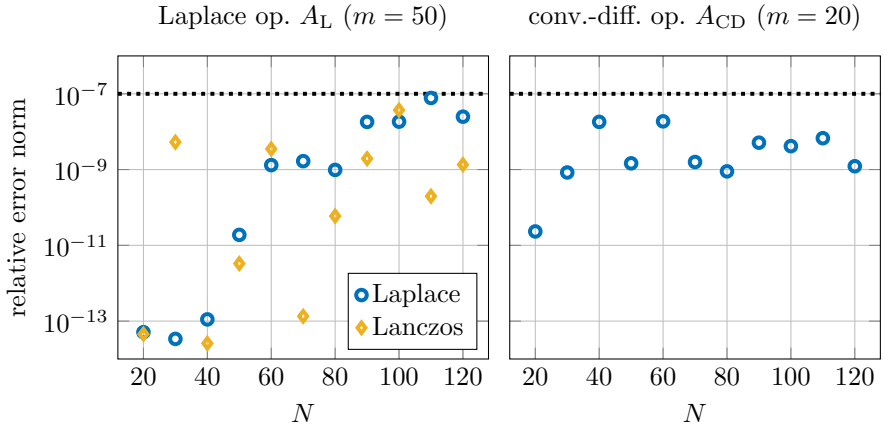

The number of matrix-vector products required is plotted in \cref{fig:gamma_matvec} and we present the achieved accuracy in \cref{fig:gamma_acc}. \corrected{For the symmetric matrix $A_L$, two-pass Lanczos requires fewer matrix-vector multiplications for $N \geq 80$.}

\subsection{Square root: $F(s) = \sqrt{s}$}
The square root is a complete Bernstein function since it can be represented as
\begin{equation*}
    F(s) = \sqrt{s} = \frac{1}{2\sqrt{\pi}} \int_0^\infty (1-\exp(-st)) t^{-3/2} \dif{t}.
\end{equation*}
The corresponding action $F(A)b=A^{1/2}b$ arises in several applications, e.g., machine learning~\cite{PleissJankowiakErikssonDamleGardner2020}, sampling from Gaussian Markov random fields~\cite{IlicTurnerSimpson2010} and preconditioning~\cite{ArioliLoghin2009}. Since
\begin{equation*}
    F'(s) = \frac{1}{2\sqrt{\pi}} \Lp{t^{-1/2}}(s)
\end{equation*}
converges absolutely for $\real{s}>0$, we can use the error representation of \cref{thm:error_laplace} for positive definite matrices, see \cref{subsec:bernstein}. For comparison, we use the package \texttt{funm\_quad} once again: Since $G(s) = s^{-1/2}$ is a Stieltjes function and $F(s) = G(s)s$, we can evaluate $F(A)b$ as $G(A)c$ with $c=Ab$.
We choose again the Laplace operator $A_\text{L}$ and the convection-diffusion operator $A_\text{CD}$ from \cref{subsec:s32}. 
The number of matrix-vector products is plotted in \cref{fig:sqrt_matvec} and the resulting error in \cref{fig:sqrt_acc}. Both plots show once more that \cref{alg:restarted_laplace} needs fewer matrix-vector products, with this time a more pronounced effect in the Hermitian case, \corrected{where also two-pass Lanczos requires more matrix-vector multiplications (except for $N=90$).}

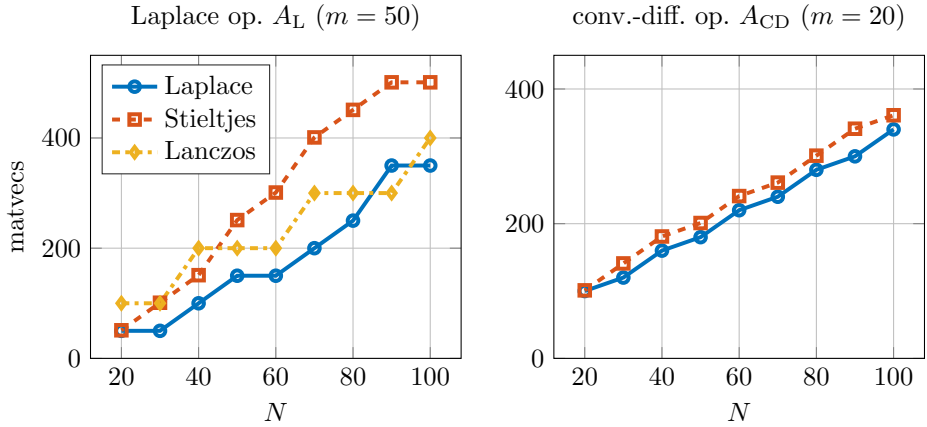
\begin{figure}\centering
\subfloat{
%
%
\begin{tikzpicture}

\begin{axis}[%
width=0.5\textwidth,
xmin=12,
xmax=108,
xlabel={$N$},
ymin=0,
ymax=550,
ylabel={matvecs},
title=\titleforLaplace,
xmajorgrids,
ymajorgrids,
legend style={at={(0.03,0.97)}, anchor=north west, legend cell align=left, align=left, draw=white!15!black}
]
\addplot [laplace1]
  table[row sep=crcr]{%
20	50\\
30	50\\
40	100\\
50	150\\
60	150\\
70	200\\
80	250\\
90	350\\
100	350\\
};
\addlegendentry{Laplace}

\addplot [stieltjes1]
  table[row sep=crcr]{%
20	51\\
30	101\\
40	151\\
50	251\\
60	301\\
70	401\\
80	451\\
90	501\\
100	501\\
};
\addlegendentry{Stieltjes}

\addplot [twopass1]
  table[row sep=crcr]{%
20	100\\
30	100\\
40	200\\
50	200\\
60	200\\
70	300\\
80	300\\
90	300\\
100	400\\
};
\addlegendentry{Lanczos}

\end{axis}
\end{tikzpicture}%
} 
\subfloat{
%
%
\begin{tikzpicture}

\begin{axis}[%
width=0.5\textwidth,
xmin=12,
xmax=108,
xlabel={$N$},
ymin=0,
ymax=450,
ylabel={\phantom{matvecs}},
title=\titleforConvdiff,
xmajorgrids,
ymajorgrids,
legend style={at={(0.03,0.97)}, anchor=north west, legend cell align=left, align=left, draw=white!15!black}
]
\addplot [laplace1]
  table[row sep=crcr]{%
20	100\\
30	120\\
40	160\\
50	180\\
60	220\\
70	240\\
80	280\\
90	300\\
100	340\\
};

\addplot [stieltjes1]
  table[row sep=crcr]{%
20	101\\
30	141\\
40	181\\
50	201\\
60	241\\
70	261\\
80	301\\
90	341\\
100	361\\
};

\end{axis}
\end{tikzpicture}%
}
\caption{Number of matrix-vector products for approximating $A^{1/2}b$. ``Laplace'' denotes \cref{alg:restarted_laplace}, ``Stieltjes'' is \texttt{funm\_quad} for $A^{-1/2}c$ with $c=Ab$. \corrected{``Lanczos'' denotes the two-pass Lanczos method.} The restart length is $m$.}
\label{fig:sqrt_matvec}
\end{figure}

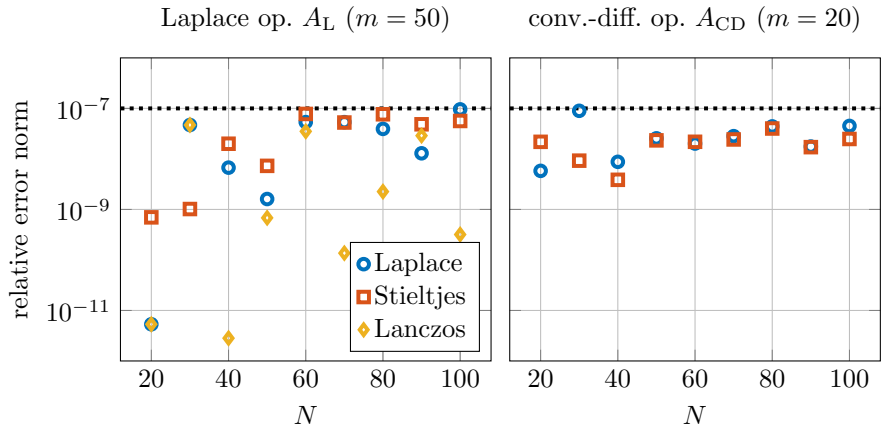
\begin{figure}\centering
\subfloat{
%
%
\begin{tikzpicture}

\begin{axis}[%
width=0.5\textwidth,
xmin=12,
xmax=108,
xlabel={$N$},
ymode=log,
ymin=1e-12,
ymax=1e-06,
ylabel={relative error norm},
title=\titleforLaplace,
xmajorgrids,
ymajorgrids,
legend style={at={(0.97,0.03)}, anchor=south east, legend cell align=left, align=left}
]
\addplot [laplace2]
  table[row sep=crcr]{%
20	5.32730584634449e-12\\
30	4.69367198004177e-08\\
40	6.71982516709966e-09\\
50	1.60720234185137e-09\\
60	5.3574813104665e-08\\
70	5.34715379750938e-08\\
80	3.9287257723336e-08\\
90	1.2910584973347e-08\\
100	9.5711742232107e-08\\
};
\addlegendentry{Laplace}

\addplot [stieltjes2]
  table[row sep=crcr]{%
20	6.96560929448992e-10\\
30	1.01816696504753e-09\\
40	1.99855082836254e-08\\
50	7.28820748948386e-09\\
60	7.77367949403069e-08\\
70	5.25364541317197e-08\\
80	7.65112361142027e-08\\
90	4.84657119897542e-08\\
100	5.64054166672059e-08\\
};
\addlegendentry{Stieltjes}

\addplot [twopass2]
  table[row sep=crcr]{%
20	5.327319841641e-12\\
30	4.69367198169612e-08\\
40	2.8246153450512e-12\\
50	6.80716087172601e-10\\
60	3.51114591990202e-08\\
70	1.36324805368665e-10\\
80	2.25326161686427e-09\\
90	2.88936411510354e-08\\
100	3.18697490103305e-10\\
};
\addlegendentry{Lanczos}

\addplot [tolline]
  table[row sep=crcr]{%
12	1e-07\\
108	1e-07\\
};
\end{axis}
\end{tikzpicture}%
} 
\subfloat{
%
%
\begin{tikzpicture}

\begin{axis}[%
width=0.5\textwidth,
xmin=12,
xmax=108,
xlabel={$N$},
ymode=log,
ymin=1e-12,
ymax=1e-06,
yticklabels={,,},
title=\titleforConvdiff,
xmajorgrids,
ymajorgrids,
legend style={at={(0.97,0.03)}, anchor=south east, legend cell align=left, align=left, draw=white!15!black}
]
\addplot [laplace2]
  table[row sep=crcr]{%
20	5.80117588096787e-09\\
30	8.96211659909018e-08\\
40	8.75609216823437e-09\\
50	2.55471232767284e-08\\
60	2.0141418190552e-08\\
70	2.8227912285824e-08\\
80	4.40077214955998e-08\\
90	1.75958309286274e-08\\
100	4.49700490905515e-08\\
};

\addplot [stieltjes2]
  table[row sep=crcr]{%
20	2.17950871839187e-08\\
30	9.23771952150424e-09\\
40	3.85944867523289e-09\\
50	2.32905835723292e-08\\
60	2.18315354883005e-08\\
70	2.42742022561312e-08\\
80	3.99850718189257e-08\\
90	1.70967750550088e-08\\
100	2.48183090259223e-08\\
};

\addplot [tolline]
  table[row sep=crcr]{%
12	1e-07\\
108	1e-07\\
};
\end{axis}
\end{tikzpicture}%
}
\caption{Accuracy at termination when approximating $A^{1/2}b$. ``Laplace'' denotes \cref{alg:restarted_laplace}, ``Stieltjes'' is \texttt{funm\_quad} for $A^{-1/2}c$ with $c=Ab$. \corrected{``Lanczos'' denotes the two-pass Lanczos method.} The restart length is $m$.}
\label{fig:sqrt_acc}
\end{figure}

\section{Conclusions}\label{sec:conclusions}
We developed a new representation of the error function in the restarted Arnoldi method for Laplace transforms. This representation allowed us to develop a new restart algorithm based on numerical integration and spline interpolation which significantly extends the class of functions for which restarting is possible in a black-box fashion without having to choose hand-tailored contours and quadrature rules depending both on $A$ and $F$.

Due to the results in \cite{Druskin2008}, we know that the error monotonically decreases with our algorithm if $A$ is Hermitian and $f$ is nonnegative. We proposed an implementation which ensures that the run-time is dominated by matrix-vector products for larger matrices, and this has been verified in numerical experiments. Some of the newly available functions are not Stieltjes functions but can still be treated with the algorithm of \cite{FrommerGuettelSchweitzer2014a} by multiplying with $A$ or $A^{-1}$. However, our experiments also illustrated that the algorithm needs fewer matrix-vector products for similar accuracy in these cases. \corrected{For the Hermitian examples, the two-pass Lanczos approach usually---but not always---required fewer matrix-vector multiplications than the new, restarted method.} Lastly, we demonstrated that our method can be applied to two-sided Laplace transforms, like the Gamma function, and complete Bernstein functions, like the square root, too.

\section*{Acknowledgments} \corrected{We would like to thank Leonid Knizhnerman and an anonymous
referee for their constructive comments which helped improve an earlier version of the manuscript.}

\appendix
\appendixnotitle \label{sec:appendix}
This appendix shows that we can interchange the order of integration in \cref{eq:order_of_integration}, i.e., in the integral
\begin{equation*}
    \int_0^\infty \int_0^\infty f(t) \exp(-\tau A) v_{m+1} g(t-\tau) u(t-\tau) \dif{\tau} \dif{t}, 
\end{equation*}
and that if $\nu$ lies within the region of absolute convergence of $\Lp{f}$, it also lies in the region of absolute convergence of $\Lp{\tilde{f}}$, which is what we left out in the proof of \cref{thm:error_laplace}.
Fubini's theorem (see, e.g., \cite[Theorem 8.8]{Rudin}) states that we can interchange the order if
\begin{equation*}
    \mathcal{I} = \int_0^\infty \int_0^\infty \abs{f(t) \exp(-\tau A)v_{m+1} g(t-\tau) u(t-\tau)} \dif{\tau} \dif{t} < \infty.
\end{equation*}

Note that the integral $\mathcal{I}$ and the integrand are vectors, so we have to show $\mathcal{I}_i < \infty$ for each entry $\mathcal{I}_i$ of $\mathcal{I}$. However, we can treat each entry in the same manner: For every vector $x \in \Cn$, it holds true that $\abs{x_i} \leq \norm{x}$. We apply this to the integrand in $\mathcal{I}$ so that we have
\begin{align*}
    \mathcal{I}_i &\leq \int_0^\infty \int_0^\infty \norm{\exp(-\tau A)v_{m+1}} \; \abs{f(t)} \; \abs{g(t-\tau)} \; u(t-\tau) \dif{\tau} \dif{t} \\
        &\leq \int_0^\infty \int_0^t \norm{\exp(-\tau A)} \; \abs{f(t)} \; \abs{g(t-\tau)} \dif{\tau} \dif{t}
\end{align*}
for every $i=1,\dots,n$. 
Similarly, we use
\begin{equation*}
    \abs{g(t)} = \abs{\Tra{e}_m \exp(-tH_m)e_1} \leq \norm{\exp(-tH_m)e_1} \leq \norm{\exp(-tH_m)}
\end{equation*}
which results in
\begin{equation} \label{eq:I_i_inequality}
    \mathcal{I}_i \leq \int_0^\infty \int_0^t \norm{\exp(-\tau A)} \; \abs{f(t)} \; \norm{\exp(-(t-\tau)H_m)} \dif{\tau} \dif{t}.
\end{equation}
Both matrix norms $\norm{\exp(-\tau A)}$ and $\norm{\exp(-(t-\tau)H_m)}$ can now be bounded by the result of Crouzeix and Palencia \cite[Theorem~3.1]{CrouzeixPalencia}. We start with
\begin{equation*}
    \norm{\exp(-\tau A)} \leq c \max_{z\in W(A)} \abs{\exp(-\tau z)} = c \max_{z\in W(A)} \exp(-\tau \real{z}) = c \exp(-\tau\nu)
\end{equation*}
with the constant $c=1+\sqrt{2}$ and $\nu = \min_{z\in W(A)}\real{z}$, where $W(A)$ is the field of values of $A$. Since $W(H_m) \subseteq W(A)$ and $t-\tau \geq 0$, one can proceed similarly for the second matrix norm, which yields
\begin{equation*}
    \norm{\exp(-(t-\tau)H_m)} \leq c \exp(-(t-\tau)\nu).
\end{equation*}
Using both bounds in \eqref{eq:I_i_inequality}, we obtain
\begin{align*}
    \mathcal{I}_i &\leq \int_0^\infty \int_0^t c^2 \exp(-\tau\nu) \abs{f(t)} \exp(-(t-\tau)\nu) \dif{\tau} \dif{t} \\
        &= c^2 \int_0^\infty \exp(-t\nu) \abs{f(t)} \int_0^t  \dif{\tau} \dif{t} \\
        &= c^2 \Lp{t\abs{f(t)}}(\nu),
\end{align*}
where the Laplace transform $-\Lp{t\abs{f(t)}}(\nu)$ is the derivative of $\Lp{\abs{f}}(\nu)$ according to \cref{thm:laplace_analytic}. By the hypothesis of \cref{thm:error_laplace}, the value $\nu$ lies within the region of absolute convergence of $\Lp{f(t)}(z)$. Thus, $-\Lp{t\abs{f(t)}}(\nu)$ is finite and, consequently,
\begin{equation*}
     \mathcal{I}_i \leq c^2 \Lp{t\abs{f(t)}}(\nu) < \infty.
\end{equation*}
In an analogous manner, one shows that
\begin{equation*}
    \Lp{\abs{\tilde{f}}}(\nu) \leq c \Lp{t\abs{f(t)}}(\nu) < \infty,
\end{equation*}
which means that $\nu$ lies in the region of absolute convergence of $\Lp{\tilde{f}}$.

\bibliography{bib,matrixfunctions}

\end{document}